\documentclass{bmcart}

\usepackage{amsmath, amscd, amsfonts, amsthm, amssymb,  graphicx}
\RequirePackage{hyperref} 
\usepackage[all]{xy}
\usepackage[utf8]{inputenc} 
\usepackage[shortlabels]{enumitem}
\setenumerate{listparindent=\parindent}

\startlocaldefs

\newtheorem{thm}{Theorem}[section]
\newtheorem{lem}[thm]{Lemma}
\newtheorem{prop}[thm]{Proposition}
\newtheorem{cor}[thm]{Corollary}
\theoremstyle{definition}

\newtheorem{defn}[thm]{Definition}
\newtheorem{examp}[thm]{Example}
\theoremstyle{remark}
\newtheorem*{rem}{Remark}

\DeclareMathOperator{\tr}{Tr}
\DeclareMathOperator{\ct}{ct}

\DeclareMathOperator{\cha}{char}
\DeclareMathOperator{\disc}{disc}
\DeclareMathOperator{\im}{im}
\DeclareMathOperator{\Hom}{Hom}

\DeclareMathOperator{\Pic}{Pic}
\DeclareMathOperator{\Sym}{Sym}
\DeclareMathOperator{\End}{End}
\DeclareMathOperator{\Pl}{Pl}
\newcommand{\<}{\left\langle}
\renewcommand{\>}{\right\rangle}
\renewcommand{\(}{\left(}
\renewcommand{\)}{\right)}
\newcommand{\GL}{\mathrm{GL}}
\newcommand{\SL}{\mathrm{SL}}
\newcommand{\CC}{\mathbb{C}}
\newcommand{\FF}{\mathbb{F}}

\newcommand{\QQ}{\mathbb{Q}}
\newcommand{\PP}{\mathbb{P}}
\newcommand{\RR}{\mathbb{R}}
\newcommand{\ZZ}{\mathbb{Z}}
\renewcommand{\aa}{\mathfrak{a}}
\newcommand{\bb}{\mathfrak{b}}
\newcommand{\cc}{\mathfrak{c}}
\newcommand{\dd}{\mathfrak{d}}
\newcommand{\pp}{\mathfrak{p}}
\newcommand{\qq}{\mathfrak{q}}
\newcommand{\OO}{\mathcal{O}}
\newcommand{\ba}{\overline}
\newcommand{\cross}{\times}
\newcommand{\tensor}{\otimes}
\newcommand{\txi}{\tilde\xi}
\newcommand{\textand}{\quad \text{and} \quad}
\newcommand{\textor}{\quad \text{or} \quad}
\renewcommand{\to}{\mathop{\rightarrow}\limits}
\newcommand{\intsec}{\cap}

\newcommand{\ignore}[1]{}

\endlocaldefs

\newcommand{\bbq}[8]{
\begin{minipage}{0.1\linewidth}
\xymatrix@!0{
& #5 \ar@{-}[rr]\ar@{-}[dd]
& & #6 \ar@{-}[dd]
\\
#1 \ar@{-}[ur]\ar@{-}[rr]\ar@{-}[dd]
& & #2 \ar@{-}[ur]\ar@{-}[dd]
\\
& #7 \ar@{-}[rr]
& & #8
\\
#3 \ar@{-}[rr]\ar@{-}[ur]
& & #4 \ar@{-}[ur]
}
\end{minipage}
}

\newdir^{ (}{{}*!/-5pt/@^{(}}
\newdir_{ (}{{}*!/-5pt/@_{(}}

\begin{document}

\begin{frontmatter}

\begin{fmbox}
\dochead{Research}


\title{Rings of small rank over a Dedekind domain and their ideals}


\author[
  addressref={home},
  email={emo916math@gmail.com}
]{\inits{EM}\fnm{Evan M} \snm{O'Dorney}}


\address[id=home]{
  \street{119 Shelterwood Lane},                     
  \postcode{94506}                            
  \city{Danville, CA},                        
  \cny{USA}                                   
}


\begin{artnotes}
\end{artnotes}

\end{fmbox}


\begin{abstractbox}

\begin{abstract} 
The aim of this paper is to find and prove generalizations of some of the beautiful integral parametrizations in Bhargava's theory of higher composition laws to the case where the base ring $\ZZ$ is replaced by an arbitrary Dedekind domain $R$. Specifically, we parametrize quadratic, cubic, and quartic algebras over $R$ as well as ideal classes in quadratic algebras, getting a description of the multiplication law on ideals that extends Bhargava's famous reinterpretation of Gauss composition of binary quadratic forms. We expect that our results will shed light on the statistical properties of number field extensions of degrees $2$, $3$, and $4$.
\end{abstract}




\end{abstractbox}
%

\end{frontmatter}


\section{Introduction}

The mathematics that we will discuss has its roots in the investigations of classical number theorists---notably Fermat, Lagrange, Legendre, and Gauss (see \cite{potf}, Ch.~I)---who were interested in what integers are represented by expressions such as $x^2 + ky^2$, for fixed $k$. It became increasingly clear that in order to answer one such question, one had to understand the general behavior of expressions of the form
\[
  ax^2 + bxy + cy^2.
\]
These expressions are now called binary quadratic forms. It was Gauss who first discovered that, once one identifies forms that are related by a coordinate change $x \mapsto px + qy, y \mapsto rx + sy$ (where $ps - qr$ = 1), the forms whose \emph{discriminant} $D = b^2 - 4ac$ has a fixed value and which are \emph{primitive}, that is, $\gcd(a,b,c) = 1$, can be naturally given the structure of an abelian group, which has the property that if forms $\phi_1, \phi_2$ represent the numbers $n_1, n_2$, then their product $\phi_1 \ast \phi_2$ represents $n_1n_2$. This group law $\ast$ is commonly called \emph{Gauss composition.}

Gauss's construction of the product of two forms was quite ad hoc. Since Gauss's time, mathematicians have discovered various reinterpretations of the composition law on binary quadratic forms, notably:
\begin{itemize}
  \item Dirichlet, who discovered an algorithm simplifying the understanding and computation of the product of two forms, which we will touch on in greater detail (see Example \ref{ex:bbz}).
  \item Dedekind, who by introducing the now-standard notion of an ideal, transformed Gauss composition into the simple operation of multiplying two ideals in a quadratic ring of the form $\ZZ[(D + \sqrt{D})/2]$;
  \item Bhargava, who in 2004 astounded the mathematical community by deriving Gauss composition from simple operations on a $2\times 2\times 2$ cube \cite{B1}.
\end{itemize}
In abstraction, Bhargava's reinterpretation is somewhat intermediate between Dirichlet's and Dedekind's: it shares the integer-based concreteness of Gauss's original investigations, yet it also corresponds to natural constructions in the realm of ideals. One of the highlights of Bhargava's method is that it extends to give group structures on objects beyond binary quadratic forms, hence the title of his paper series, ``Higher composition laws.'' It also sheds light on previously inaccessible conjectures about Gauss composition, such as an estimate for the number of forms of bounded discriminant whose third power is the identity \cite{BV}.

A second thread that will be woven into this thesis is the study of finite ring extensions of $\ZZ$, often with a view toward finite field extensions of $\QQ$. Quadratic rings (that is, those having a $\ZZ$-basis with just two elements) are simply and classically parametrized by a single integer invariant, the \emph{discriminant}. For cubic rings, Delone and Faddeev prove a simple lemma (as one of many tools for studying irrationalities of degree $3$ and $4$ over $\QQ$) parametrizing them by binary cubic forms (\cite{DF}, pp.~101ff). A similar classification for quartic and higher rings proved elusive until Bhargava, using techniques inspired by representation theory, was able to parametrize quartic and quintic rings together with their cubic and sextic \emph{resolvent} rings, respectively, and thereby compute the asymptotic number of quartic and quintic rings and fields with bounded discriminant \cite{B3,Bd4c,B4,Bd5c}. The analytic virtue of Bhargava's method is to map algebraic objects such as rings and ideals to lattice points in bounded regions of $\RR^n$, where asymptotic counting is much easier. (Curiously enough, the ring parametrizations seem to reach a natural barrier at degree $5$, in contrast to the classical theory of solving equations by radicals where degree $4$ is the limit.)

Bhargava published these results over the integers $\ZZ$. Since then, experts have wondered whether his techniques apply over more general classes of rings; by far the most ambitious extensions of this sort are Wood's classifications of quartic algebras \cite{WQuartic} and ideals in certain $n$-ic algebras \cite{W2xnxn} over an arbitrary base scheme $S$. In this paper, all results are proved over an arbitrary Dedekind domain $R$. The use of a Dedekind domain has the advantage of remaining relevant to the original application (counting number fields and related structures) while introducing some new generality.

We will focus on two parametrizations that are representative of Bhargava's algebraic techniques in general. The first is the famous reinterpretation of Gauss composition in terms of $2 \times 2 \times 2$ boxes. Following \cite{B1}, call a triple $(I_1,I_2,I_3)$ of ideals of a quadratic ring $S$ \emph{balanced} if $I_1I_2I_3 \subseteq S$ and $N(I_1)N(I_2)N(I_3) = 1$, and call two balanced triples \emph{equivalent} if $I_i = \gamma_i I'_i$ for some scalars $\gamma_i \in S \tensor_\ZZ \QQ$ having product $1$. (If $S$ is Dedekind, as is the most common application, then the balanced triples of equivalence classes correspond to triples of ideal classes having product $1$.) Then:

\begin{thm}[\cite{B1}, Theorem 11]
There is a canonical bijection between
\begin{itemize}
  \item pairs $(S,(I_1,I_2,I_3))$ where $S$ is an oriented quadratic ring of nonzero discriminant over $\ZZ$ and $(I_1,I_2,I_3)$ is an equivalence class of balanced triples of ideals of $S$;
  \item trilinear maps $\beta : \ZZ^2 \tensor \ZZ^2 \tensor \ZZ^2 \to \ZZ$, up to $\SL_2 \ZZ$-changes of coordinates in each of the three inputs (subject to a certain nondegeneracy condition).
\end{itemize}
\end{thm}

Our parametrization is analogous, with one crucial difference. Whereas over $\ZZ$, the only two-dimensional lattice is $\ZZ^2$, over a Dedekind domain $R$ there are as many as there are ideal classes, and any such lattice can serve as the $R$-module structure of a quadratic algebra or an ideal thereof. Using a definition of balanced and equivalent essentially identical to Bhargava's (see Definition \ref{defn:balanced}), we prove:
\begin{thm}[see Theorem \ref{thm:box}]
Let $R$ be a Dedekind domain. There is a canonical bijection between
\begin{itemize}
  \item pairs $(S,(I_1,I_2,I_3))$ where $S$ is an oriented quadratic algebra over $R$ and $(I_1,I_2,I_3)$ is an equivalence class of balanced triples of ideals of $S$;
  \item quadruples $(\aa, (M_1,M_2,M_3), \theta, \beta)$ where $\aa$ is an ideal class of $R$, $M_i$ are lattices of rank $2$ over $R$ (up to isomorphism), $\theta : \Lambda^2 M_1 \tensor \Lambda^2 M_2 \tensor \Lambda^2 M_3 \to \aa^3$ is an isomorphism, and $\beta : M_1 \tensor M_2 \tensor M_3 \to \aa$ is a trilinear map whose three partial duals $\beta_i : M_j \tensor M_k \to \aa M_i^*$ $(\{i,j,k\} = \{1,2,3\})$ have image a full-rank sublattice.
\end{itemize}
Under this bijection, we get identifications $\Lambda^2 S \cong \aa$ and $I_i \cong M_i$.
\end{thm}
In particular $R$ may have characteristic $2$, the frequent factors of $1/2$ in Bhargava's exposition notwithstanding, and by weakening the nondegeneracy condition, we are able to include balanced triples in degenerate rings.

The second main result of our paper is the parametrization of quartic rings (with the quadratic and cubic parametrizations as preliminary cases). A key insight is to parametrize not merely the quartic rings themselves, but the quartic rings together with their \emph{cubic resolvent} rings, a notion arising from the resolvent cubic used in the classical solution of the quartic by radicals.
\begin{thm}[\cite{B3}, Theorem 1 and Corollary 5]
There is a canonical bijection between
\begin{itemize}
  \item isomorphism classes of pairs $(Q,C)$ where $Q$ is a quartic ring (over $\ZZ$) and $C$ is a cubic resolvent ring of $Q$;
  \item quadratic maps $\phi : \ZZ^3 \to \ZZ^2$, up to linear changes of coordinates on both the input and the output.
\end{itemize}
Any quartic ring $Q$ has a cubic resolvent, and if $Q$ is Dedekind, the resolvent is unique.
\end{thm}
Our analogue is as follows:
\begin{thm}[see Theorems \ref{thm:quartic} and \ref{thm:cub ring struc} and Corollary \ref{cor:at least 1 res}]
Let $R$ be a Dedekind domain. There is a canonical bijection between
\begin{itemize}
  \item isomorphism classes of pairs $(Q,C)$ where $Q$ is a quartic ring (over $R$) and $C$ is a cubic resolvent ring of $Q$;
  \item quadruples $(L,M,\theta,\phi)$ where $L$ and $M$ are lattices of ranks $3$ and $2$ over $R$ respectively, $\theta : \Lambda^2 M \to \Lambda^3 L$ is an isomorphism, and $\phi : L \to M$ is a quadratic map.
\end{itemize}
Under this bijection, we get identifications $Q/R \cong L$ and $C/R \cong M$.

Any quartic ring $Q$ has a cubic resolvent, and if $Q$ is Dedekind, the resolvent is unique.
\end{thm}

\subsection{Outline}
The remainder of this paper is structured as follows. In section \ref{sec:ded}, we set up basic definitions concerning projective modules over a Dedekind domain. In sections \ref{sec:quad} and \ref{sec:quadideals}, respectively, we generalize to Dedekind base rings two classical parametrizations, namely of quadratic algebras over $\ZZ$ and of their ideals. In section \ref{sec:box}, we prove Bhargava's parametrization of balanced ideal triples (itself a generalization of Gauss composition) over a Dedekind domain. In section \ref{sec:p-adic}, we work out in detail a specific example---unramified extensions of $\ZZ_p$---that allows us to explore the notion of balanced ideal triple in depth. In sections \ref{sec:cubic} and \ref{sec:quartic}, we tackle cubic and quartic algebras respectively, and in section \ref{sec:maximality}, we discuss results that would be useful when using the preceding theory to parametrize and count quartic field extensions.

\section{Modules and algebras over a Dedekind domain}
\label{sec:ded}

A \emph{Dedekind domain} is an integral domain that is Noetherian, integrally closed, and has the property that every nonzero prime ideal is maximal. The standard examples of Dedekind domains are the ring of algebraic integers $\OO_K$ in any finite extension $K$ of $\QQ$; in addition, any field and any principal ideal domain (PID), such as the ring $\CC[x]$ of polynomials in one variable, is Dedekind. In this section, we summarize properties of Dedekind domains that we will find useful; for more details, see \cite{Milnor}, pp.~9--18.

The salient properties of Dedekind domains were discovered through efforts to generalize prime factorization to rings beyond $\ZZ$; in particular, every nonzero ideal $\aa$ in a Dedekind domain $R$ is expressible as a product $\pp_1^{a_1}\cdots \pp_n^{a_n}$ of primes, unique up to ordering. Our motivation for using Dedekind domains stems from two other related properties. Recall that a \emph{fractional ideal} or simply an \emph{ideal} of $R$ is a finitely generated nonzero $R$-submodule of the fraction field $K$ of $R$, or equivalently, a set of the form $a\aa$ where $\aa \subseteq R$ is a nonzero ideal and $a \in K^\cross$. (The term ``ideal'' will from now on mean ``(nonzero) fractional ideal''; if we wish to speak of ideals in the ring-theoretic sense, we will use a phrasing such as ``ideal $\aa \subseteq R$.'') The first useful property is that any fractional ideal $\aa \subseteq K$ has an inverse $\aa^{-1}$ such that $\aa\aa^{-1} = R$. This allows us to form the group $I(R)$ of nonzero fractional ideals and quotient by the group $K^\cross/R^\cross$ of principal ideals to obtain the familiar \emph{ideal class group,} traditionally denoted $\Pic R$. (For the ring of integers in a number field, the class group is always finite; for a general Dedekind domain this may fail, e.g.~for the ring $\CC[x,y]/(y^2 - (x-a_1)(x-a_2)(x-a_3))$ of functions on a punctured elliptic curve.)

The second property that we will find very useful is that finitely generated modules over a Dedekind domain are classified by a simple theorem generalizing the classification of finitely generated abelian groups. For our purposes it suffices to discuss torsion-free modules, which we will call lattices.
\begin{defn}
Let $R$ be a Dedekind domain and $K$ its field of fractions. A \emph{lattice} over $R$ is a finitely generated, torsion-free $R$-module $M$. If $M$ is a lattice, we will denote by the subscript $M_K$ its $K$-span $M \tensor_R K$ (except when $M$ is denoted by a symbol containing a subscript, in which case a superscript will be used). The dimension of $M_K$ over $K$ is called the \emph{rank} of the lattice $M$.
\end{defn}
A lattice of rank $1$ is a nonzero finitely generated submodule of $K$, i.e.~an ideal; thus isomorphism classes of rank-$1$ lattices are parametrized by the class group $\Pic R$. The situation for general lattices is not too different.
\begin{thm}[see \cite{Milnor}, Lemma 1.5, Theorem 1.6, and the intervening Remark] \label{thm:lattices}
A lattice $M$ over $R$ is classified up to isomorphism by two invariants: its rank $m$ and its top exterior power $\Lambda^m M$. Equivalently, every lattice is a direct sum $\aa_1 \oplus \cdots \oplus \aa_m$ of nonzero ideals, and two such direct sums $\aa_1 \oplus \cdots \oplus \aa_m$, $\bb_1 \oplus \cdots \oplus \bb_n$ are isomorphic if and only if $m = n$ and the products $\aa_1\cdots\aa_m$ and $\bb_1\cdots\bb_n$ belong to the same ideal class.
\end{thm}

In this paper we will frequently be performing multilinear operations on lattices. Using Theorem \ref{thm:lattices}, it is easy to show that these operations behave much more ``tamely'' than for modules over general rings. Specifically, for two lattices $M = \aa_1u_1 \oplus \cdots \oplus \aa_mu_m$ and $N = \bb_1v_1 \oplus \cdots \oplus \bb_nv_n$, we can form the following lattices:
\begin{itemize}
  \item the tensor product
  \[
    M \tensor N = \bigoplus_{\substack{1 \leq i \leq m \\ 1 \leq j \leq n}}
      \aa_i\bb_j (u_i \tensor v_j);
  \]
  \item the symmetric powers
  \[
    \Sym^k M = \bigoplus_{1 \leq i_1 \leq \cdots \leq i_k \leq m} \aa_{i_1} \cdots \aa_{i_k}(u_{i_1} \tensor \cdots \tensor u_{i_k})
  \]
  and the exterior powers
  \[
    \Lambda^k M = \bigoplus_{1 \leq i_1 < \cdots < i_k \leq m} \aa_{i_1} \cdots \aa_{i_k}(u_{i_1} \wedge \cdots \wedge u_{i_k})
  \]
  of ranks $\binom{n+k-1}{k}$ and $\binom{n}{k}$ respectively;
  \item the dual lattice
  \[
    M^* = \Hom(M, R) = \bigoplus_{1 \leq i \leq m} \aa_i^{-1} u_i^*;
  \]
  \item and the space of homomorphisms
  \[
    \Hom(M, N) \cong M^* \tensor N = \bigoplus_{\substack{1 \leq i \leq m \\ 1 \leq j \leq n}}
      \aa_i^{-1}\bb_j (u_i^* \tensor v_j).
  \]  
\end{itemize}
A particular composition of three of these constructions is of especial relevance to the present thesis:
\begin{defn}
If $M$ and $N$ are lattices (or, more generally, $M$ is a lattice and $N$ is any $R$-module), then a degree-$k$ \emph{map} $\phi : M \to N$ is an element of $(\Sym^k M^*) \otimes N$. A map to a lattice $N$ of rank $1$ is called a \emph{form}.
\end{defn}
In terms of decompositions $M = \aa_1u_1 \oplus \cdots \oplus \aa_mu_m$ and $N = \bb_1v_1 \oplus \cdots \oplus \bb_nv_n$, a degree-$k$ map can be written in the form
\[
  \phi(x_1u_1 + \cdots + x_mu_m) = \sum_{j=1}^n \sum_{i_1 + \cdots + i_m = k} a_{i_1,\ldots,i_m,j} \cdot x_1^{i_1} \cdots x_m^{i_m} v_j,
\]
where the coefficients $a_{i_1,\ldots,i_m,j}$ belong to the ideals $\aa_1^{-i_1}\cdots \aa_m^{-i_m} \bb_j$ needed to make each term's value belong to $N$. For example, over $R = \ZZ$, a quadratic map from $\ZZ^2$ to $\ZZ$ is a quadratic expression
\[
  \phi(x,y) = ax^2 + bxy + cy^2
\]
in the coordinates $x,y \in (\ZZ^2)^*$ on $\ZZ^2$. Two caveats about this notion are in order:
\begin{itemize}
  \item Although such a degree-$k$ map indeed yields a function from $M$ to $N$ (evaluated by replacing every functional in $M^*$ appearing in the map by its value on the given element of $M$), it need not be unambiguously determined by this function if $R$ is finite. For instance, if $R = \FF_2$ is the field with two elements, the cubic map from $\FF_2^2$ to $\FF_2$ defined by $\phi(x,y) = xy(x+y)$ vanishes on each of the four elements of $\FF_2^2$, though it is not the zero map.
  \item Also, one must not confuse $(\Sym^k M^*) \otimes N$ with the space $(\Sym^k M)^* \otimes N$ of symmetric $k$-ary multilinear functions from $M$ to $N$. Although both lattices have rank $n\binom{m+k-1}{k}$ and there is a natural map from one to the other (defined by evaluating a multilinear function on the diagonal), this map is not in general an isomorphism. For instance, the quadratic forms $\phi : \ZZ^2 \to \ZZ$ arising from a symmetric bilinear form $\lambda((x_1,y_1),(x_2,y_2)) = ax_1x_2 + b(x_1y_2 + x_2y_1) + cy_1y_2$ are exactly those of the form $\phi(x,y) = ax^2 + 2bxy + cy^2$, with middle coefficient even.
\end{itemize}

\begin{defn}\label{defn:image}
  The \emph{image} $\phi(M)$ of a degree-$k$ map $\phi : M \to N$ is the smallest sublattice $N' \subseteq N$ such that $\phi$ is a degree-$k$ map from $M$ to $N'$, i.e.~lies in the image of the natural injection $(\Sym^k M^*) \otimes N' \hookrightarrow (\Sym^k M^*) \otimes N$. It may be computed as follows: if
\[
  \phi(x_1u_1 + \cdots + x_mu_m) = \sum_{i_1 + \cdots + i_m = k} x_1^{i_1} \cdots x_m^{i_m} \cdot v_{i_1,\ldots,i_m},
\]
then $\phi(M)$ is the $R$-span of all the $1$-dimensional sublattices $\aa_1^{i_1}\cdots \aa_m^{i_m} v_{i_1,\ldots,i_m}$ in $N$. (It is \emph{not} the same as the span of the values of $\phi$ as a function on $M$.)
\end{defn}

\begin{defn}
  If $L \subseteq M$ are two lattices of rank $n$, the \emph{index} $[M:L]$ is the ideal $\aa$ such that
\[
  \aa \cdot \Lambda^n L = \Lambda^n M.
\]
Since $\Lambda^n L$ and $\Lambda^n M$ are of rank $1$, this is well defined.
\end{defn}

\subsection{Algebras}
An \emph{algebra of rank $n$} over $R$ is a lattice $S$ of rank $n$ equipped with a multiplication operation giving it the structure of a (unital commutative associative) $R$-algebra. Since $R$ is integrally closed, the sublattice generated by $1\in S$ must be primitive (that is, the lattice it generates is maximal for its dimension, and therefore a direct summand of $S$), implying that the quotient $S/R$ is a lattice of rank $n-1$ and we have a noncanonical decomposition
\begin{equation} \label{eq:alg decomp}
  S = R \oplus S/R.
\end{equation}
We will be concerned with algebras of ranks $2$, $3$, and $4$, which we call quadratic, cubic, and quartic algebras (or rings) respectively.

\subsection{Orientations}
When learning about Gauss composition over $\ZZ$, one must sooner or later come to a problem that vexed Legendre (see \cite{potf}, p.~42): If one considers quadratic forms up to $\GL_2 \ZZ$-changes of variables, then a group structure does not emerge because the conjugate forms $ax^2 \pm bxy + cy^2$, which ought to be inverses, have been identified. Gauss's insight was to consider forms only up to ``proper equivalence,'' i.e.~$\SL_2 \ZZ$ coordinate changes. This is tantamount to considering quadratic forms not simply on a rank-$2$ $\ZZ$-lattice $M$, but on a rank-$2$ $\ZZ$-lattice equipped with a distinguished generator of its top exterior power $\Lambda^2 M$. For general lattices over Dedekind domains, whose top exterior powers need not belong to the principal ideal class, we make the following definitions.
\begin{defn}
Let $\aa$ be a fractional ideal of $R$. A rank-$n$ lattice $M$ is of \emph{type $\aa$} if its top exterior power $\Lambda^n M$ is isomorphic to $\aa$; an \emph{orientation} on $M$ is then a choice of isomorphism $\alpha : \Lambda^n M \to \aa$. The possible orientations on any lattice $M$ are of course in noncanonical bijection with the units $R^\cross$. The easiest way to specify an orientation on $M$ is to choose a decomposition $M = \bb_1u_1 \oplus \cdots \oplus \bb_nu_n$, where the ideals $\bb_i$ are scaled to have product $\aa$, and then declare
\[
  \alpha(y_1u_1 \wedge \cdots \wedge y_nu_n) = y_1\cdots y_n.
\]
An orientation on a rank-$n$ $R$-algebra $S$ is the same as an orientation on the lattice $S$, or equivalently on the lattice $S/R$, due to the isomorphism between $\Lambda^n S$ and $\Lambda^{n-1} S/R$ given by
\[
  1 \wedge v_1 \wedge \cdots \wedge v_{n-1} \mapsto \tilde{v}_1 \wedge \cdots \wedge \tilde{v}_{n-1}.
\]
(Here, and henceforth, we use a tilde to denote the image under the quotient map by $R$, so that the customary bar can be reserved for conjugation involutions. This is opposite to the usual convention where $\tilde{v}$ denotes a lift of $v$ under a quotient map.)
\end{defn}

\section{Quadratic algebras}
\label{sec:quad}
Before proceeding to Bhargava's results, we lay down as groundwork two parametri\-zations that, over $\ZZ$, were known classically. These are the parametrizations of quadratic algebras and of ideal classes in quadratic algebras. The extension of these to other base rings has been thought about extensively, with many different kinds of results produced (see \cite{WGauss} and the references therein). Here, we prove versions over a Dedekind domain that parallel our cubic and quartic results.

Let $S$ be a quadratic algebra over $R$. Since $S/R$ has rank $1$, the decomposition \eqref{eq:alg decomp} simplifies to $S = R \oplus \aa\xi$ for an (arbitrary) ideal $\aa$ in the class of $\Lambda^2 S$ and some formal generator $\xi \in S_K$. The algebra is then determined by $\aa$ and a multiplication law $\xi^2 = t\xi - u$, which allows us to describe the ring as $R[\aa\xi]/(\aa^2(\xi^2 - t\xi + u))$, a subring of $K[\xi]/(\xi^2 - t\xi + u)$. Alternatively, we can associate to the ring its norm map
\[
  N_{S/R} : S \to R, \quad x + y\xi \mapsto x^2 + t x y + u y^2.
\]
It is evident that this is just another way of packaging the same data, namely two numbers $t \in \aa^{-1}$ and $u \in \aa^{-2}$. The norm map is more readily freed from coordinates than the multiplication table, yielding the following parametrization.
\begin{lem}
Quadratic algebras over $R$ are in canonical bijection with rank-$2$ $R$-lattices $M$ equipped with a distinguished copy of $R$ and a quadratic form $\phi : M \to R$ that acts as squaring on the distinguished copy of $R$.
\end{lem}
\begin{proof}
Given $M$ and $\phi$, the distinguished copy of $R$ must be primitive (otherwise $\phi$ would take values outside $R$), yielding a decomposition $M = R \oplus \aa\xi$. Write $\phi$ in these coordinates as
\[
  \phi(x + y\xi) = x^2 + t x y + u y^2;
\]
then the values $t \in \aa^{-1}$ and $u \in \aa^{-2}$ can be used to build a multiplication table on $M$ having the desired norm form (which is unique, as for any fixed coordinate system, the norm form determines $t$ and $u$, which determine the multiplication table).
\end{proof}
If there is a second copy of $R$ on which $N_{S/R}$ restricts to the squaring map, it must be generated by a unit of $S$ with norm $1$, multiplication by which induces an automorphism of the lattice with norm form. Hence we can eliminate the distinguished copy of $R$ and arrive at the following arguably prettier parametrization:
\begin{thm}\label{thm:nice ring}
Quadratic algebras over $R$ are in canonical bijection with rank-$2$ $R$-lattices $M$ equipped with a quadratic form $\phi : M \to R$ attaining the value $1$.
\end{thm}

For our applications to Gauss composition it will also be helpful to have a parametrization of \emph{oriented} quadratic algebras. An orientation $\alpha: \Lambda^2 S \to \aa$ can be specified by choosing an element $\xi$ with $\alpha(1 \wedge \xi) = 1$. Since $\xi$ is unique up to translation by $\aa^{-1}$, the parametrization is exceedingly simple.
\begin{thm}
For each ideal $\aa$ of $R$, there is a canonical bijection between oriented quadratic algebras of type $\aa$ and pairs $(t,u)$, where $t \in \aa^{-1}$, $u \in \aa^{-2}$, up to the action of $\aa^{-1}$ via
\[
  s.(t,u) = (t + 2s, u + st + s^2)
\]
\end{thm}

One other fact that will occasionally be useful is that every quadratic algebra has an involutory automorphism defined by $\bar{x} = \tr x - x$ or, in a coordinate representation
\[
  S = R[\aa\xi]/(\aa^2(\xi^2 - t\xi + u)),
\]
by $\xi \mapsto t - \xi$. (The first of these characterizations shows that the automorphism is well-defined, the second that it respects the ring structure.)
\begin{examp}
When $R = \QQ$ (or more generally any Dedekind domain in which $2$ is a unit), then completing the square shows that oriented quadratic algebras are in bijection with the forms $x^2 - ky^2$, $k \in \QQ$, each of which yields an algebra $S = \QQ[\sqrt{k}]$ oriented by $\alpha(1 \wedge \sqrt{k}) = 1$.

If we pass to \emph{unoriented} extensions, then we identify $\QQ[\sqrt{k}]$ with its rescalings $\QQ[f\sqrt{k}] \cong \QQ[\sqrt{f^2k}]$, $f \in \QQ^\cross$. The resulting orbit space $\QQ/(\QQ^\cross)^2$ parametrizes quadratic number fields, plus the two nondomains
\[
  \QQ[\sqrt{0}] = \QQ[\epsilon]/(\epsilon^2) \textand \QQ[\sqrt{1}] \cong \QQ \oplus \QQ.
\]
\end{examp}

\begin{examp}
When $R = \ZZ$, we can almost complete the square, putting a general $x^2 + txy + uy^2$ in the form
\[
  x^2 - \frac{D}{4} y^2 \textor x^2 + xy - \frac{D-1}{4} y^2.
\]
Here $D = t^2 - 4u$ is the \emph{discriminant}, the standard invariant used in \cite{B1} to parametrize oriented quadratic rings. It takes on all values congruent to $0$ or $1$ mod $4$. It also parametrizes \emph{unoriented} quadratic rings, since each such ring has just two orientations which are conjugate under the ring's conjugation automorphism. The rings of integers of number fields are then parametrized by the \emph{fundamental discriminants} which are not a square multiple of another discriminant, with the exception of $0$ and $1$ which parametrize $\ZZ[\epsilon]/\epsilon^2$ and $\ZZ \oplus \ZZ$ respectively.
\end{examp}

\begin{examp}
For an example where discriminant-based parametrizations are inapplicable, consider the field $R = \FF_2$ of two elements. Any nonzero quadratic form attains the value $1$, and there are three such, namely
\[
  x^2, \quad xy, \textand x^2 + xy + y^2.
\]
They correspond to the three quadratic algebras over $\FF_2$, respectively $\FF_2[\epsilon]/\epsilon^2$, $\FF_2 \oplus \FF_2$, and $\FF_4$.
\end{examp}

\section{Ideal classes of quadratic algebras}
\label{sec:quadideals}
We can now parametrize ideal classes of quadratic algebras, in a way that partially overlaps \cite{WGauss}. To be absolutely unambiguous, we make the following definition for quadratic algebras that need not be domains:
\begin{defn}
Let $S$ be a quadratic algebra over $R$. A \emph{fractional ideal} (or just an \emph{ideal}) of $S$ is a finitely generated $S$-submodule of $S_K$ \textbf{that spans $S_K$ over} $K$. Two fractional ideals are considered to belong to the same \emph{ideal class} if one is a scaling of the other by a scalar $\gamma \in S_K^\cross$. (This is clearly an equivalence relation.) The ideal classes together with the operation induced by ideal multiplication form the \emph{ideal class semigroup}, and the invertible ideal classes form the \emph{ideal class group} $\Pic S$.
\end{defn}
The condition in bold means that, for instance, the submodule $R \oplus \{0\} \subseteq R \oplus R$ is not a fractional ideal. Of course, any ideal that is invertible automatically satisfies it.
\begin{thm}[cf.~\cite{WGauss}, Corollary 4.2] \label{thm:quadideals}
For each ideal $\aa$ of $R$, there is a bijection between
\begin{itemize}
  \item ideal classes of oriented quadratic rings of type $\aa$, and
  \item rank-$2$ lattices $M$ equipped with a nonzero quadratic map $\phi : M \to \aa^{-1} \cdot \Lambda^2 M$.
\end{itemize}
In this bijection, the ideal classes that are invertible correspond exactly to the forms that are \emph{primitive,} that is, do not factor through any proper sublattice of $\aa^{-1} \cdot \Lambda^2 M$.
\end{thm}
\begin{proof}
Suppose first that we have a quadratic ring $S = R \oplus \aa\xi$, oriented by $\alpha(1 \wedge \xi) = 1$, and a fractional ideal $I$ of $R$. Construct a map $\phi : I \to \aa^{-1} \cdot \Lambda^2 I$ by
\[
  \omega \mapsto \omega \wedge \xi \omega.
\]
Here $\xi \omega \in \aa^{-1}I$ so the wedge product lies in $\aa^{-1} \cdot \Lambda^2 I$, and we get a well-defined quadratic map $\phi$, scaling appropriately when $I$ is scaled by an element of $S_K^\cross$. Note that $\phi$ is nonzero because, after extending scalars to $K$, the element $1 \in I_K = S_K$ is mapped to $1 \wedge \xi \neq 0$.

It will be helpful to write this construction in coordinates. Let $I = \bb_1\eta_1 \oplus \bb_2\eta_2$ be a decomposition into $R$-ideals, and let $\xi$ act on $I$ by the matrix
$\begin{bmatrix}
  a & b \\
  c & d
\end{bmatrix}$, that is,
\begin{equation} \label{eq:act on ideal}
\begin{aligned}
  \xi\eta_1 &= a\eta_1 + c\eta_2 \\
  \xi\eta_2 &= b\eta_1 + d\eta_2
\end{aligned}
\end{equation}
where $a,b,c,d$ belong to the relevant ideals: $a,d \in \aa^{-1}$, $b \in \aa^{-1}\bb_1\bb_2^{-1}$, and $c \in \aa^{-1}\bb_1^{-1}\bb_2$. Then we get
\begin{equation} \label{eq:qf}
\begin{aligned}
  \phi(x\eta_1 + y\eta_2) &= (x\eta_1 + y\eta_2) \wedge (x\xi\eta_1 + y\xi\eta_2) \\
  &= (x\eta_1 + y\eta_2) \wedge (ax\eta_1 + cx\eta_2 + by\eta_1 + dy\eta_2) \\
  &= (cx^2 + (d-a)xy - by^2)(\eta_1 \wedge \eta_2) \in \aa^{-1}\bb_1\bb_2(\eta_1 \wedge \eta_2) = \aa^{-1} \Lambda^2 I.
\end{aligned}
\end{equation}
(Now $\phi$ appears clearly as a tensor in $\Sym^2 I^* \tensor \aa^{-1} \cdot \Lambda^2 M$.)

We now seek to reconstruct the ideal $I$ from its associated quadratic form. Given an ideal $\aa$, a lattice $M = \bb_1\eta_1 \oplus \bb_2\eta_2$, and a quadratic map $\phi(x\eta_1 + y\eta_2) = (px^2 + qxy + ry^2)(\eta_1 \wedge \eta_2)$ to $\aa^{-1} \cdot \Lambda^2 M$, we may choose $a=0$, $b=-r$, $c = p$, and $d = q$ to recover an action \eqref{eq:act on ideal} of $\xi$ on $R$ yielding the form $\phi$. By \eqref{eq:qf}, this action is unique up to adding a constant to $a$ and $d$, which simply corresponds to a change of basis $\xi \mapsto \xi+a$. Next, by the Cayley-Hamilton theorem, the formal expression $\xi^2 - q\xi + pr$ annihilates $M$, so $M$ is a module over the ring $S = R[\aa\xi]/(\aa^{2}(\xi^2 - q\xi + pr))$ corresponding to the quadratic form $x^2 + qxy + pry^2$. The last step is to embed $M$ into $S_K$, or equivalently, to identify $M_K$ with $S_K$. For this, we divide into cases based on the kind of ring that $S_K$ is, or equivalently the factorization type of the polynomial $f(x) = x^2 - qx + pr$ over $K$.
\begin{itemize}
  \item If $f$ is irreducible, then $S_K$ is a field, and $M_K$ is an $S_K$-vector space of dimension $1$, isomorphic to $S_K$.
  \item If $f$ has two distinct roots, then $S_K \cong K \oplus K$. There are three different $S_K$-modules  having dimension $2$ as $K$-vector spaces: writing $I_1$ and $I_2$ for the two copies of $K$ within $S_K$, we can describe them as $I_1 \oplus I_1$, $I_2 \oplus I_2$, and $I_1 \oplus I_2$. But on the first two, every element of $S_K$ acts as a scalar. If $M_K$ were one of these, then the quadratic form $\phi(\omega) = \omega \wedge \xi \omega$ would be identically $0$, which is not allowed. So $M_K \cong I_1 \oplus I_2 \cong S_K$.
  \item Finally, if $f$ has a double root, then $S_K \equiv K[\epsilon]/\epsilon^2$. There are two $S_K$-modules having dimension $2$ as a $K$-vector space: $K\epsilon \oplus K\epsilon$ and $S_K$. On $K\epsilon \oplus K\epsilon$, $S_K$ acts by scalars and we get a contradiction as before. So $M_K \cong S_K$.
\end{itemize}
This shows that there is always at least one embedding of $M$ into $S_K$. To show there is at most one up to scaling, we need that every automorphism of $S_K$ as an $S_K$-module is given by multiplication by a unit. But this is trivial (the image of $1$ determines everything else).

It will be convenient to have as well an explicit reconstruction of an ideal from its associated quadratic form. First change coordinates on $M$ such that $p \neq 0$. (If $r \neq 0$, swap $\bb_1\eta_1$ and $\bb_2\eta_2$; if $p = 0$ but $q \neq 0$, translate $\eta_2 \mapsto \eta_2 + t\eta_1$ for any nonzero $t \in \bb_1\bb_2^{-1}$.) Then the ideal
\begin{equation} \label{eq:expl ideal}
  I = \bb_1 + \bb_2\left(\frac{\xi}{p}\right)
\end{equation}
of the ring $S = R[\aa\xi]/(\aa^2(\xi^2 - q\xi + pr))$ corresponding to the norm form $x^2 + qxy + pry^2$ is readily seen to yield the correct quadratic form.

We now come to the equivalence between invertibility of ideals and primitivity of forms. Suppose first that $\phi : M \to \aa^{-1} \cdot \Lambda^2 M$ is imprimitive, that is, there is an ideal $\aa'$ strictly containing $\aa$ such that $\phi$ actually arises from a quadratic map $\phi' : M \to \aa'^{-1} \cdot \Lambda^2 M$. Following through the (first) construction, we see that $\phi$ and $\phi'$ give the same $\xi$-action on $I = M$ but embed it as a fractional ideal in two different rings, $S = R \oplus \aa\xi$ and $S' = R \oplus \aa'\xi$. We naturally have $S_K \cong S'_K \cong K[\xi]/(\xi^2 - q\xi + pr)$, and $S$ is a subring of $S'$. Suppose $I$ had an inverse $J$ as an $S$-ideal. Then since $I$ is an $S'$-ideal, the product $IJ = S$ must be an $S'$-ideal, which is a contradiction.

Conversely, suppose that $\phi$ is primitive and $I$ has been constructed using \eqref{eq:expl ideal}. Consider the conjugate ideal
\[
  \bar{I} = \bb_1 + \bb_2 \frac{\bar{\xi}}{p} = \bb_1 + \bb_2 \frac{q - \xi}{p}
\]
and form the product
\begin{align*}
  I\bar{I} &= \left(\bb_1 + \bb_2 \frac{\xi}{p}\right)\left(\bb_1 + \bb_2 \frac{q - \xi}{p}\right) \\
  &= \bb_1^2 + \bb_1\bb_2\frac{\xi}{p} + \bb_1\bb_2\frac{q - \xi}{p} + \bb_2^2 \frac{\xi\bar{\xi}}{p^2} \\
  &= \frac{1}{p}(p \bb_1^2 + q \bb_1\bb_2 + r \bb_2^2 + \xi\bb_1\bb_2).
\end{align*}
The first three terms in the parenthesis are all fractional ideals in $K$. The condition that $\phi$ maps into $\aa^{-1} \cdot \Lambda^2 I$ is exactly that these lie in $\aa^{-1}\bb_1\bb_2$, and the condition of primitivity is that they do not all lie in any smaller ideal, that is, their sum is $\aa^{-1}\bb_1\bb_2$. So
\begin{equation} \label{eq:conjideal}
  I\bar{I} = \frac{\bb_1\bb_2}{p} (\aa^{-1} + R\xi) = \frac{\aa^{-1}\bb_1\bb_2}{p} \cdot S.
\end{equation}
We conclude that
\[
  I^{-1} = \aa\bb_1^{-1}\bb_2^{-1}p\bar{I} = \aa \alpha(\Lambda^2 I)^{-1} \bar{I}
\]
is an inverse for $I$.
\end{proof}
Note that our proof of the invertibility-primitivity equivalence shows something more: that \emph{any} fractional ideal $I$ of a quadratic algebra $S$ is invertible when considered as an ideal of a certain larger ring $S'$, found by ``canceling common factors'' in its associated quadratic form. The following relation is worth noting:
\begin{cor}\label{cor:norm}
If $I$ is an ideal of a quadratic algebra $S$ and $S' = \End I \subseteq S_K$ is its ring of endomorphisms, then
\[
  I\bar{I} = \frac{\alpha(\Lambda^2I)}{\alpha(\Lambda^2 S')} \cdot S'.
\]
\end{cor}
\begin{proof}
The ring $S'$ is the one occurring in the proof that imprimitivity implies noninvertibility, provided that the ideal $\aa'$ is chosen to be as large as possible (i.e.~equal to $(p \bb_1^2 + q \bb_1\bb_2 + r \bb_2^2)^{-1}$), so that $I$ is actually invertible with respect to $S'$. This $S'$ must be the endomorphism ring $\End I$, or else $I$ would be an ideal of an even larger quadratic ring. (We here need that $\End I$ is finitely generated and hence a quadratic ring. This is obvious, as it is contained in $x^{-1}I$ for any $x \in S_K^\cross \cap I$.)

Viewing $\alpha$, by restriction, as an orientation on $S'$, we have $\alpha(\Lambda^2 S') = \aa'$ and the formula is reduced to that for $I^{-1}$ above.
\ignore{
The conclusion follows from \eqref{eq:conjideal} and the computation
\[
  \Lambda^2 I = \bb_1 \wedge \left(\bb_2 \frac{\xi}{p}\right) = \frac{\bb_1\bb_2}{p} (1 \wedge \xi) = \alpha^{-1}\left(\frac{\aa^{-1}\bb_1\bb_2}{p}\right).
\]
}
\end{proof}

\begin{examp} \label{ex:5i}
If $R = \ZZ$ (or more generally any PID), then the situation simplifies to $\aa = \ZZ$ and $M = \ZZ^2$, and we recover a bijection between ideal classes and binary quadratic forms. But the theorem also requires us, when changing coordinates on $M$, to change coordinates on $\Lambda^2 M$ appropriately; that is, ideal classes are in bijection with $\GL_2(\ZZ)$-orbits of binary quadratic forms $\phi : \ZZ^2 \to \ZZ$, not under the natural action but under the twisted action
\[
  \left(\begin{bmatrix}
    a & b \\ c & d
  \end{bmatrix}
  \mathop{.\vphantom{I}} \phi\right)(x,y) = \frac{1}{ad-bc} \cdot \phi(ax + cy, bx + dy).
\]
(Compare \cite{potf}, p.~142 and \cite{WGauss}, Theorem 1.2.)

For an example not commonly encountered in the literature, take the order $S = \ZZ[5i]$ in the domain $\ZZ[i]$. Its ideal classes correspond simply to $\GL_2(\ZZ)$-orbits of quadratic forms $px^2 + qxy + ry^2$ having discriminant $q^2 - 4pr = -100$. Using the standard theory of ``reduction'' of quadratic forms developed by Lagrange (see \cite{potf}, pp.~26ff.), we may limit our search to the bounded domain where $|q| \leq r \leq p$ and find that there are precisely three, with three corresponding ideal classes:
\begin{alignat*}{3}
  \phi_1(x,y) &= x^2 + 25y^2         &&\leftrightsquigarrow{}& S &= \ZZ[5i] \\
  \phi_2(x,y) &= 2x^2 + 2xy + 13y^2\,&&\leftrightsquigarrow{}& A &= {\ZZ}{\<5, 1+i\>} \\
  \phi_3(x,y) &= 5x^2 + 5y^2         &&\leftrightsquigarrow{}& B &= \ZZ[i].
\end{alignat*}
The first two ideals, which correspond to primitive forms, are invertible (indeed $A \cdot iA = S$); the third is not. In fact we can build a multiplication table for the ideal class semigroup.
\[
\begin{tabular}{c|ccc}
$\cdot$ & $S$ & $A$ & $B$ \\ \hline
$S$ & $S$ & $A$ & $B$ \\
$A$ & $A$ & $S$ & $B$ \\
$B$ & $B$ & $B$ & $B$
\end{tabular}
\]
\end{examp}

\section{Ideal triples}
\label{sec:box}
We turn now to one of Bhargava's most widely publicized contributions to mathematics, the reinterpretation of Gauss's 200-year-old composition law on primitive binary quadratic forms in terms of simple operations on a $2\times 2\times 2$ box of integers. In fact, Bhargava produced something rather more general: a bijection (\cite{B1}, Theorem~1) that takes \emph{all} $2\times 2\times 2$ boxes satisfying a mild nondegeneracy condition, up to the action of the group
\[
  \Gamma = \left\{(M_1, M_2, M_3) \in (\GL_2\ZZ)^3 : \prod_i \det M_i = 1\right\},
\]
to triples of fractional ideals $(I_1,I_2,I_3)$ in a quadratic ring $S$ that are \emph{balanced,} that is, satisfy the two conditions
\begin{enumerate}[$($a$)$]
  \item $I_1 I_2 I_3 \subseteq S$;
  \item $N(I_1)N(I_2)N(I_3) = 1$. Here $N(I)$ is the norm of the ideal $I$, defined by the formula $N(I) = [A:I]/[A:S]$ for any $\ZZ$-lattice $A$ containing both $S$ and $I$. (This should not be confused with the ideal generated by the norms of the elements of $I$. Even over $\ZZ$, the two notions differ: $2\cdot \ZZ[i]$ is an ideal of norm $2$ in the ring $\ZZ[2i]$, but every element of $2\cdot \ZZ[i]$ has norm divisible by $4$.)
\end{enumerate}
The ideals $I_i$ are unique up to a scaling by constants $\gamma_i \in S_\QQ^\cross$ of product $1$.

Our task will be to generalize this result to an arbitrary Dedekind domain. First, the definition of balanced extends straightforwardly, provided that we define the norm of a fractional ideal $I$ properly, as the \emph{index} of $I$ in $S$ as an $R$-lattice. The resulting notion of balanced is a special case of the definition used in \cite{W2xnxn}:
\begin{defn} \label{defn:balanced}
A triple of fractional ideals $I_1, I_2, I_3$ of an $R$-algebra $S$ is \emph{balanced} if
\begin{enumerate}[$($a$)$]
  \item $I_1 I_2 I_3 \subseteq S$;
  \item the image of $\Lambda^2 I_1 \tensor \Lambda^2 I_2 \tensor \Lambda^2 I_3$ in $(\Lambda^2 S_K)^{\tensor 3}$ is precisely $(\Lambda^2 S)^{\tensor 3}$.
\end{enumerate}
\end{defn}
The objects that we will use on the other side of the bijection are, as one might expect, not merely $8$-tuples of elements from $R$, because the class group intrudes. The appropriate notion is as follows:
\begin{defn} \label{defn:box}
Let $\aa$ be an ideal class of $R$. A \emph{Bhargava box} of type $\aa$ over $R$ consists of the following data:
\begin{itemize}
  \item three rank-$2$ lattices $M_1$, $M_2$, $M_3$;
  \item an orientation isomorphism $\theta : \Lambda^2 M_1 \tensor \Lambda^2 M_2 \tensor \Lambda^2 M_3 \to \aa^3$;
  \item a trilinear map $\beta : M_1 \tensor M_2 \tensor M_3 \to \aa$ satisfying the following nondegeneracy condition: each of the three partial duals $\beta_i : M_j \tensor M_k \to \aa M_i^*$ $(\{i,j,k\} = \{1,2,3\})$ has image a full-rank sublattice.
\end{itemize}
\end{defn}
If we choose a decomposition of each $M_i$ into a direct sum $\bb_{i1} \oplus \bb_{i2}$ of ideals, then $\theta$ becomes an isomorphism from $\prod_{i,j} \bb_{ij}$ to $\aa^3$ (which we may take to be the identity), while $\beta$ is determined by eight coefficients
\[
  \beta_{ijk} \in \bb_{1i}^{-1}\bb_{2j}^{-1}\bb_{3k}^{-1}\aa.
\]
Thus we stress that, in spite of all the abstraction, our parameter space indeed still consists of (equivalence classes of) $2\times 2\times 2$ boxes of numbers lying in certain ideals contained in $K$.
\begin{thm}[cf.~\cite{B1}, Theorem 1; \cite{W2xnxn}, Theorem 1.4]\label{thm:box}
For each ideal $\aa$ of $R$, there is a bijection between
\begin{itemize}
  \item balanced triples $(I_1,I_2,I_3)$ of ideals in an oriented quadratic ring $S$ of type $\aa$, up to scaling by factors $\gamma_1$, $\gamma_2$, $\gamma_3 \in S_K^\cross$ with product $1$;
  \item Bhargava boxes of type $\aa$.
\end{itemize}
\end{thm}
\begin{rem}
Two balanced ideal triples may be inequivalent for the purposes of this bijection even if corresponding ideals belong to the same class (see Example \ref{ex:bbz}\ref{item:5i box}). Consequently a Bhargava box cannot be described as corresponding to a balanced triple of ideal \emph{classes.}
\end{rem}
\begin{proof}
The passage from ideals to the Bhargava box is simple and derived directly from \cite{B1}. Given a balanced triple $(I_1,I_2,I_3)$ in a quadratic ring $S$ with an orientation $\alpha : \Lambda^2 S \to \aa$, construct the trilinear map
\begin{align*}
  \beta : I_1 \tensor I_2 \tensor I_3 &\to \aa \\
  x \tensor y \tensor z &\mapsto \alpha(1 \wedge xyz).
\end{align*}
This, together with the identification $\theta$ coming from condition (b) of Definition \ref{defn:balanced}, furnishes the desired Bhargava box. Since each $I_i$ spans $S_K$, the nondegeneracy is not hard to check.

We seek to invert this process and reconstruct the ring $S$, the orientation $\alpha$, and the ideals $I_i$ uniquely from the Bhargava box. We begin by reconstructing the quadratic forms $\phi_i : M_i \to \aa^{-1}  \cdot \Lambda^2 M_i$ corresponding to the ideals $I_i$. For this we first use $\beta$ to map $M_1$ to $\Hom(M_2 \tensor M_3,\aa)$, in other words $\Hom(M_2, \aa M_3^*)$. We then take the determinant, which lands us in $\Hom(\Lambda^2 M_2, \Lambda^2(\aa M_3^*)) \cong \aa^2 \cdot \Lambda^2 M_2^* \tensor \Lambda^2 M_3^*$, which can be identified via $-\theta$ (note the sign change) with $\aa^{-1} \Lambda^2 M_1$. We thus get a quadratic form $\phi'_1 : M_1 \to \aa^{-1} \Lambda^2 M_1$. We claim that if the Bhargava box arose from a triple of ideals, then this is the natural form $\phi_1 : x \mapsto x \wedge \xi x$ on $I_1$. For convenience we will extend scalars and prove the equality as one of forms on $M_1^K \cong S_K$. To deal with $\phi'_1$, we must analyze
\[
  \beta(x) = (y \mapsto (z \mapsto \alpha(1\wedge xyz))) \in \Hom(M_2^K, M_3^{K*}).
\]
Now whereas $M_2^K$ is naturally identifiable with $S_K$, to deal with $M_3^{K*} \cong S_K^*$ we have to bring in the symmetric pairing $\alpha(1 \wedge \bullet\bullet) : S_K \tensor_K S_K \to K$, which one easily checks is nondegenerate and thus identifies $S_K^*$ with $S_K$. So we have transformed $\beta(x)$ to the element 
\[
  \beta'(x) = (y \mapsto xy) \in \Hom_K(S_K, S_K).
\]
We then take the determinant $\det \beta'(x)$, which is simply the norm $N(x) \in K \cong \Hom_K(\Lambda^2 S_K,{} \Lambda^2 S_K)$. This is to be compared to
\[
  \phi_1(x) = x \wedge \xi x = N(x) (1 \wedge \xi) = \alpha^{-1}(N(x)).
\]
It then remains to check that we have performed the identifications properly, that is, that the four isomorphisms
\[
\xymatrix{
  K & \Lambda^2(M_1^K \tensor_{S_K} M_2^K) \ar[l]_\alpha \ar[d]^{\wedge^2(x\tensor y \mapsto \alpha(xy\bullet))} \\
  \Lambda^2M_1^K \tensor \Lambda^2M_2^K \ar[r]^{-\theta} \ar[u]^{\alpha \tensor \alpha} & \Lambda^2M_3^{K*}
}
\]
are compatible. In particular we discover that the pairing $\alpha(1 \wedge \bullet\bullet)$ is given in the basis $\{1, \xi\}$ by the matrix
\[
  \begin{bmatrix}
  0 & 1 \\
  1 & \tr \xi
  \end{bmatrix}
\]
of determinant $-1$, explaining the compensatory minus sign that must be placed on $\theta$.

Now write $M_i = \bb_{i1}\eta_{i1} \oplus \bb_{i2}\eta_{i2}$ where $\theta : \prod_{i,j} \bb_{ij} \to \aa^3$ may be assumed to be the identity map, and express $\beta$ in these coordinates as
\[
  \beta\left(\sum_{i,j,k} x_{ijk}\eta_{1i}\eta_{2j}\eta_{3k}\right) = \sum_{i,j,k} a_{ijk}x_{ijk}.
\]
It will be convenient to create the single-letter abbreviations $a = a_{111}$, $b = a_{112}$, $c = a_{121}$, continuing in lexicographic order to $h = a_{222}$. Then $\phi_1$ sends an element $x\eta_{11} + y\eta_{12} \in M_1$ to the determinant
\[
  -\det\begin{bmatrix}
    ax + ey & bx + fy \\
    cx + gy & dx + hy
  \end{bmatrix} = (bc - ad)x^2 + (bg + cf - ah - de)xy + (fg - eh)y^2.
\]
We claim that $\phi_1 \neq 0$. If not, the linear maps from $M_2^K$ to $M_3^{K*}$ corresponding to every element of $M_1^K$ are singular. It is not hard to prove that a linear system with dimension at most $2$ of singular maps from $K^2$ to $K^2$ has either a common kernel vector or images in a common line, and to deduce from this that the partial dual $M_1^K \tensor M_3^K \to M_2^{K*}$ or $M_1^K \tensor M_2^K \to M_3^{K*}$, respectively, is not surjective, a contradiction.

Thus $M_1$ can be equipped with the structure of a fractional ideal of some quadratic ring, with a $\xi$-action given by the matrix
\begin{equation} \label{eq:xi action}
  \begin{bmatrix}
    ah + de & eh - fg \\
    bc - ad & bg + cf
  \end{bmatrix}
\end{equation}
where we have added a scalar matrix such that the trace $ah + bg + cf + de$, and indeed the entire characteristic polynomial
\begin{equation} \label{eq:long ring}
  F(x) = x^2 - (ah + bg + cf + de)x + abgh + acfh + adeh + bcfg + bdeg + cdef - adfg - bceh,
\end{equation}
is symmetric under permuting the roles of $M_1$, $M_2$, and $M_3$. In other words, we have exhibited a single ring $S = R[\aa\xi]/\aa^2F(\xi)$ over which $M_1$, $M_2$, and $M_3$ are modules, under the $\xi$-action \eqref{eq:xi action} and its symmetric cousins
\[
  \begin{bmatrix}
  ah + cf & ch - dg \\
  be - af & bg + de
  \end{bmatrix} 
  \text{ on $M_2$ and }
  \begin{bmatrix}
  ah + bg & bh - df \\
  ce - ag & cf + de
  \end{bmatrix}
  \text{ on $M_3$.}
\]

The next step is is the construction of the elements $\tau_{ijk}$ that will serve as the products $\eta_{1i}\eta_{2j}\eta_{3k}$ of the ideal generators. Logically, it begins with a ``voil\`a'' (compare \cite{B1}, p.~235):
\[
  \tau_{ijk} = \begin{cases}
    a_{\bar{i}jk}a_{i\bar{j}k}a_{ij\bar{k}} + a_{ijk}^2a_{\bar{i}\bar{j}\bar{k}} - a_{ijk}\bar{\xi}, &
    i+j+k \text{ odd,} \\
    -a_{\bar{i}jk}a_{i\bar{j}k}a_{ij\bar{k}} - a_{ijk}^2a_{\bar{i}\bar{j}\bar{k}} + a_{ijk}\xi, &
    i+j+k \text{ even.}
  \end{cases}
\]
Here $\bar{i}$, $\bar{j}$, $\bar{k}$ are shorthand for $3-i$, etc., while $\bar{\xi}$ denotes the Galois conjugate $\tr(\xi) - \xi$. Bhargava apparently derived this formula (in the case $R = \ZZ$) by solving the natural system of quadratic equations that the $\tau$'s must satisfy ($\tau_a\tau_d = \tau_b\tau_c$ and so on). For our purposes it suffices to note that this formula is well-defined over any Dedekind domain (in contrast to \cite{B1} where there is a denominator of $2$) and yields a trilinear map $\tilde{\beta} : M_1\tensor M_2\tensor M_3 \to S$, defined by
\[
  \tilde{\beta}\left(\sum_{i,j,k} x_{ijk}\eta_{1i}\eta_{2j}\eta_{3k}\right) = \sum_{i,j,k} \tau_{ijk}x_{ijk},
\]
with the property that following with the projection $\alpha(1 \wedge \bullet) : S \to \aa$ gives back $\beta$. We claim that $\tilde{\beta}$, in addition to being $R$-trilinear, is $S$-trilinear under the newfound $S$-actions on the $M_i$. This is a collection of calculations involving the action of $\xi$ on each factor, for instance
\[
  (ah + de)\tau_a + (bc-ad)\tau_e = \xi\tau_a
\]
(where we have taken the liberty of labeling the $\tau_{ijk}$ as $\tau_a,\ldots,\tau_h$ in the same manner as the $a_{ijk}$). This is routine, and all the other edges of the box can be dealt with symmetrically. So, extending scalars to $K$, we get a map
\[
  \tilde{\beta} : M_1^K \tensor_{S_K} M_2^K \tensor_{S_K} M_3^K \to S_K.
\]
Since each $M_i$ is isomorphic to a fractional ideal, each $M_i^K$ is isomorphic to $S_K$ and thus so is the left side. Also, it is easy to see that $\tilde{\beta}$ is surjective or else $\beta$ would be degenerate. So once two identifications $\iota_1 : M_1 \to I_1$, $\iota_2 : M_2 \to I_2$ are chosen, the third $\iota_3 : M_3 \to I_3$ can be scaled such that
$\tilde{\beta}(x \tensor y \tensor z) = \iota_1(x)\iota_2(y)\iota_3(z)$ and hence $\beta(x \tensor y \tensor z) = \alpha(1 \wedge \iota_1(x)\iota_2(y)\iota_3(z))$ is as desired.

We have now constructed a triple $(I_1,I_2,I_3)$ of fractional ideals such that the map $\alpha(1\wedge \bullet\bullet\bullet) : I_1 \tensor I_2 \tensor I_3 \to K$ coincides with $\beta$. Two verifications remain:
\begin{itemize}
  \item That $I_1 I_2 I_3 \subseteq S$. Since $I_1I_2I_3$ is the $R$-span of the eight $\bb_{1i}\bb_{2j}\bb_{3k}\tau_{ijk}$, this is evident from the construction of the $\tau_{ijk}$.
  \item That $\prod_i \Lambda^2(I_i) = \prod_i \Lambda^2(S)$, and more strongly that the diagram
  \[
    \xymatrix{
      \bigotimes_i \Lambda^2(M_i) \ar[r]^{\prod_i \iota_i} \ar[dr]_{\theta} &
      \bigotimes_i \Lambda^2(I_i) \ar[d]^{\alpha^{\tensor 3}} \\
      & K
    }
  \]
  commutes. This is a verification similar to that which showed the correspondence of the forms $\phi_i$. Indeed, if we had recovered a triple of ideals that produced the correct $\beta$ but the wrong $\theta$, then the $\phi$'s as computed from $\beta$ and the two $\theta$'s would have to mismatch.
\end{itemize}

This concludes the proof that each Bhargava box corresponds to at least one balanced triple. We must also prove that two balanced triples $(I_1,I_2,I_3)$ and $(I'_1, I'_2, I'_3)$ yielding the same Bhargava box must be equivalent; but here we are helped greatly by the results that we have already proved. Namely, since the forms $\phi_i$ associated to the ideals match, these ideals must lie in the same oriented quadratic ring $S$ and there must be scalars $\gamma_i \in S_K^\cross$ such that $I'_i = \gamma_i I_i$. We may normalize such that $\gamma_2 = \gamma_3 = 1$. Then, for all $x \in I_1, y \in I_2, z \in I_3$,
\[
  0 = \beta(xyz) - \beta(xyz) = \alpha(1 \wedge xyz) - \alpha(1 \wedge \gamma_1 xyz) = \alpha(1 \wedge (1-\gamma_1)xyz).
\]
In other words, we have $(1-\gamma_1)x \in K$ for every $x \in I_1I_2I_3$. Extending scalars, we get the same for all $x \in KI_1I_2I_3 = S_K$ which implies $1 - \gamma = 0$.
\end{proof}

\subsection{Relation with the class group}
Just as in the case $R = \ZZ$, we can restrict to invertible ideals and get a new description of the class group.
\begin{thm}[cf. \cite{B1}, Theorem 1]
Let $\aa$ be an ideal of $R$, and let $G$ be the set of rank-$2$ lattices $M$ equipped with a primitive quadratic form $\phi : M \to \aa^{-1} \cdot \Lambda^2 M$, up to isomorphism. Then the relations
\begin{itemize}
  \item $\phi_1 \ast \phi_2 \ast \phi_3 = 1$ for all $(\phi_1,\phi_2,\phi_3)$ arising from a Bhargava box;
  \item $\phi = 1$ if $\aa^{-1} \cdot \Lambda^2 M$ is principal and $\phi$ attains a generator of it
\end{itemize}
give $G$ the structure of a disjoint union of abelian groups. That is, if we partition $G$ into equivalence classes under the relation that $\phi_1 \sim \phi_2$ if $\phi_1$ and $\phi_2$ are two of the three forms arising from one Bhargava box, then each equivalence class gains the structure of an abelian group. These groups are isomorphic to the class groups of all quadratic extensions of $R$ of type $\aa$ under the bijection of Theorem \ref{thm:quadideals}. 
\end{thm}
\begin{proof}
It is easy to see that a triple $(I_1,I_2,I_3)$ of invertible ideals in a ring $S$ is balanced if and only if $I_1I_2I_3 = S$. Each $\sim$-equivalence class in the theorem is the family of forms corresponding to the ideals in a single ring, since we showed that the three forms arising from one Bhargava box belong to the same ring, and conversely if $I_1$ and $I_2$ belong to the same ring then $(I_1,I_2,I_1^{-1}I_2^{-1})$ is balanced (which also shows that $\sim$ is truly an equivalence relation).

The condition that $\phi$ attains a generator of $\aa^{-1} \cdot \Lambda^2 M$ simply says that $\phi$ matches the form corresponding to the entire ring $S$ itself in Theorem \ref{thm:nice ring}, which is also the form corresponding to the principal class in Theorem \ref{thm:quadideals}. Now the theorem is reduced to the elementary fact that the structure of an abelian group is determined by the triples of elements that sum to $0$, together with the identification of that $0$-element (without which any $3$-torsion element could take its place).
\end{proof}

After establishing the corresponding theorem in \cite{B1} establishing a group law on quadratic forms, Bhargava proceeds to Theorem 2, which establishes a group law on the $2\times 2\times 2$ cubes themselves, or rather on the subset of those that are ``projective,'' i.e.~correspond to triples of invertible ideals. This structure is easily replicated in our situation: it is only necessary to note that the product of two balanced triples of invertible ideals is balanced. In fact, a stronger result holds.

\begin{lem}
Let $(I_1,I_2,I_3)$ and $(J_1,J_2,J_3)$ be balanced triples of ideals of a quadratic ring $S$, with each $I_i$ invertible. Then the ideal triple $(I_1J_1,I_2J_2,I_3J_3)$ is also balanced.
\end{lem}
\begin{proof}
We clearly have
\[
  I_1J_1 \cdot I_2J_2 \cdot I_3J_3 = (I_1I_2I_3)(J_1J_2J_3) \subseteq S,
\]
establishing (a) of Definition \ref{defn:balanced}. For (b), the key is to use Corollary \ref{cor:norm} to get a handle on the exterior squares of the $I_iJ_i$. We have $\End I_i = S$; each $S_i = \End J_i$ is a quadratic ring with $S \subseteq S_i \subseteq S_K$. Then since
\[
  \End J_i \subseteq \End I_i J_i \subseteq \End I_i^{-1}I_i J_i = \End J_i,
\]
we see that $\End I_i J_i = S_i$ as well. Then
\[
  \frac{\alpha(\Lambda^2(I_iJ_i))}{\alpha(S_i)} S_i = I_iJ_i \cdot \ba{I_iJ_i} = I_i\ba{I_i} \cdot J_i\ba{J_i}
  = \alpha(\Lambda^2 I_i) S \cdot \frac{\alpha(\Lambda^2 J_i)}{\alpha(S_i)} S_i
  = \frac{\alpha(\Lambda^2 I_i) \alpha(\Lambda^2 J_i)}{\alpha(S_i)} S_i.
\]
Intersecting with $K$, we get
\[
  \alpha(\Lambda^2(I_iJ_i)) = \alpha(\Lambda^2 I_i) \alpha(\Lambda^2 J_i).
\]
We can now multiply and get
\[
  \prod_i \alpha(\Lambda^2(I_iJ_i)) = \prod_i \alpha(\Lambda^2 I_i) \cdot \prod_i \alpha(\Lambda^2 J_i) = R,
\]
so $(I_1J_1,I_2J_2,I_3J_3)$ is balanced.
\end{proof}
\begin{cor}[cf.~\cite{B1}, Theorems 2 and 12]\label{cor:box group}
The Bhargava boxes which belong to a fixed ring $S$ (determined by the quadratic form \eqref{eq:long ring}) and which are primitive (in the sense of having all three associated quadratic forms primitive) naturally form a group isomorphic to $(\Pic S)^2$.
\end{cor}
\begin{cor}
The Bhargava boxes which belong to a fixed ring $S$ naturally have an action by $(\Pic S)^2$.
\end{cor}

It is natural to think about what happens when the datum $\theta$ is removed from the Bhargava box. As one easily verifies, multiplying $\theta$ by a unit $u \in R^\cross$ is equivalent to multiplying the orientation $\alpha$ of $S$ by $u^{-1}$ while keeping the same ideals $I_i$. Accordingly, we have the following corollary, which we have chosen to state with a representation-theoretic flavor:
\begin{cor}
Balanced triples of ideals $(I_1,I_2,I_3)$ of types $\aa_1$, $\aa_2$, $\aa_3$ in an (unoriented) quadratic extension $S$ of type $\aa$, up to equivalence, are parametrized by $\GL(M_1) \cross \GL(M_2) \cross \GL(M_3)$-orbits of trilinear maps
\[
  \beta : M_1 \tensor M_2 \tensor M_3 \to \aa,
\]
where $M_i$ is the module $R \oplus \aa_i$, satisfying the nondegeneracy condition of Definition \ref{defn:box}.
\end{cor}
These orbits do \emph{not} have a group structure. Indeed, the identifications cause a box and its inverse, under the group law of Corollary \ref{cor:box group}, to become identified.

\ignore{
Now the equivalence classes of balanced triples of invertible ideals form a group isomorphic to $(\Pic S)^2$, since $I_1$ and $I_2$ are arbitrary, and then $I_3$ is forced to equal $I_1^{-1}I_2^{-1}$. So we get an action of $P = (\Pic S)^2$ on the Bhargava boxes with associated ring $S$. One naturally wonders what the orbits are. The following theorem provides a pretty answer.
\begin{thm}
Two balanced triples $(I_1,I_2,I_3)$ and $(J_1,J_2,J_3)$ of ideals of a quadratic ring $S$ are in the same $P$-orbit if and only if
\begin{itemize}
  \item $\End I_i = \End J_i$ for each $i$ (equivalently, the corresponding quadratic forms have the same ``content,'' i.e.~gcd of their coefficients)
  \item $I_1I_2I_3 = J_1J_2J_3$.
\end{itemize}
\end{thm}
\begin{proof}
The $P$-invariance of $I_1I_2I_3$ is obvious. As for $\End I_i$, note that on multiplying by an invertible $P_i$ it can only grow, and on multiplying by $P_i^{-1}$ it can only grow again. So the heart of the proof is the converse, namely, that two triples $(I_1,I_2,I_3)$ and $(J_1,J_2,J_3)$ satisfying the conditions above are indeed in the same $P$-orbit.
\end{proof}
}

\begin{examp} \label{ex:bbz}
When $R = \ZZ$ (or more generally any PID), we can simplify the notation of a Bhargava box by taking each $M_i = \ZZ^2$, so that $\theta$ is without loss of generality the standard orientation $\Lambda^2(\ZZ^2)^{\tensor 3} \to^\sim \ZZ$, and $\beta$ is expressible as a box
\[
  \bbq{a}{b}{c}{d}{e}{f}{g}{h}
\]
of integers. The three forms $\phi_i$ are then obtained by slicing $\beta$ into two $2\times 2$ matrices and taking the determinant of a general linear combination as described in \cite{B1}, Section 2.1:
\[
  \phi_1(x,y) = -\det\left(
    x\begin{bmatrix}
      a & b \\ c & d
    \end{bmatrix} +
    y\begin{bmatrix}
      e & f \\ g & h
    \end{bmatrix}
  \right).
\]
We can now derive a balanced triple of ideals from any box of eight integers $a,b,\ldots,h$, subject only to the very mild condition that no two opposite faces should be linearly dependent. We recapitulate the boxes having the greatest significance in \cite{B1} and in the theory of quadratic forms generally:
\begin{enumerate}[(a)]
\item The boxes
\[
  \bbq{0}{1}{1}{0}{1}{0}{0}{D/4} \textand \bbq{0}{1}{1}{1}{1}{1}{1}{(D+3)/4}
\]
(for $D$ even and odd respectively), have as all three of their associated quadratic forms $x^2 - (D/4)y^2$ and $x^2 + xy - (D-1)/4\cdot y^2$ respectively, the defining form of the ring $S$ of discriminant $D$. They correspond to the balanced triple $(S,S,S)$. These are the ``identity cubes'' of \cite{B1}, equation (3).
\item The boxes
\[
  \bbq{0}{1}{1}{0}{a}{-b/2}{b/2}{-c} \textand \bbq{0}{1}{1}{0}{a}{(-b+1)/2}{(b+1)/2}{-c}
\]
(for $b$ even and odd respectively), have as two of their associated quadratic forms the conjugates
\[
  ax^2 + bxy + cy^2 \textand ax^2 - bxy + cy^2
\]
and as the third associated form the form $x^2 - (D/4)y^2$ or $x^2 + xy - (D-3)/4\cdot y^2$ defining the ring $S$ of discriminant $D = b^2 - 4ac$. These boxes express the fact that the triple
\[
  (S, I, \alpha(\Lambda^2 I)^{-1} \bar{I})
\]
is always balanced (compare Corollary \ref{cor:norm}). If $\gcd(a,b,c) = 1$, we also get that $I$ and $\bar{I}$ represent inverse classes in the class group and that, correspondingly, $ax^2 + bxy + cy^2$ and $ax^2 - bxy + cy^2$ are inverse under Gauss's composition law on binary quadratic forms.
\item The box
\[
  \bbq{1}{0}{0}{d}{0}{f}{g}{-h}
\]
has as associated quadratic forms
\begin{align*}
  \phi_1(x,y) &= -dx^2 + hxy + fgy^2 \\
  \phi_2(x,y) &= -gx^2 + hxy + dfy^2 \\
  \phi_3(x,y) &= -fx^2 + hxy + dgy^2.
\end{align*}
As Bhargava notes (\cite{B1}, p.~249), Dirichlet's simplification of Gauss's composition law was essentially to prove that any pair of primitive binary quadratic forms of the same discriminant can be put in the form $(\phi_1,\phi_2)$, so that the multiplication relation that we derive from this box,
\[
  \phi_1 \ast \phi_2 = -fx^2 - hxy + dgy^2 \text{ (or, equivalently, }dgx^2 + hxy - fy^2),
\]
encapsulates the entire multiplication table for the class group.
\item \label{item:5i box} For some examples not found in the classical theory of primitive forms, we consider the non-Dedekind domain $S = \ZZ[5i]$, whose ideal class semigroup was computed above (Example \ref{ex:5i}). Let us find all balanced triples that may be formed from the ideals
\[
  S = \ZZ[5i], \quad
  A = {\ZZ}{\<5, 1+i\>}, \quad
  B = \ZZ[i]
\]
of $S$. We compute
\[
  \alpha(\Lambda^2 S) = \ZZ, \quad \alpha(\Lambda^2 A) = \ZZ, \quad \alpha(\Lambda^2 B) = \frac{1}{5}\ZZ.
\]
For each triple $(I_1,I_2,I_3)$ of ideal class representatives, finding all balanced triples of ideals in these classes is equivalent to searching for all $\gamma \in S_K^\cross$ satisfying $\gamma \cdot I_1I_2I_3 \subseteq S$ which have the correct norm
\[
  \<N(\gamma)\> = \frac{1}{\alpha(\Lambda^2 I_1) \cdot \alpha(\Lambda^2 I_2) \cdot \alpha(\Lambda^2 I_3)}
\]
(the right side is an ideal of $\ZZ$, so $N(\gamma)$ is hereby determined up to sign, and as we are in a purely imaginary field, $N(\gamma) > 0$).

Using the class $B$ zero or two times, we get four balanced triples
\[
  (S,S,S), \quad (S,A,iA), \quad (S,B,5B), \textand (A,B,5B),
\]
each of which yields one Bhargava box. We get no balanced triples involving the ideal class $B$ just once; indeed, it is not hard to show in general that if two ideals of a balanced triple are invertible, so is the third.

The most striking case is $I_1 = I_2 = I_3 = B$, for here there are two multipliers $\gamma$ of norm $125$ that send $B^3 = \ZZ[i]$ into $\ZZ[5i]$, namely $10 + 5i$ and $10 - 5i$ (we could also multiply these by powers of $i$, but this does not change the ideal $B$). The balanced triples $(B,B,(10+5i)B)$ and $(B,B,(10-5i)B$ are inequivalent under scaling, although corresponding ideals belong to the same classes. Thus we get two inequivalent Bhargava boxes with the same three associated forms, namely
\[
  \bbq{1}{2}{2}{-1}{2}{-1}{-1}{-2} \textand \bbq{-1}{2}{2}{1}{2}{1}{1}{-2}.
\]
\item The triply symmetric boxes
\[
  \bbq{a}{b}{b}{c}{b}{c}{c}{d}
\]
correspond to balanced triples of ideals that all lie in the same class; those that are \emph{projective}---that is, whose associated forms are primitive---correspond to invertible ideal classes whose third power is the trivial class. This correspondence was used to prove estimates for the average size of the $3$-torsion of class groups in \cite{BV}. Our work suggests that similar methods may work for quadratic extensions of rings besides $\ZZ$.
\end{enumerate}
\end{examp}

\section{Another example: \texorpdfstring{$p$}{p}-adic rings}
\label{sec:p-adic}
\begin{examp}
It is instructive to look at the local rings $R = \ZZ_p$, where for simplicity we assume $p \geq 3$. Thanks to the large supply of squares, the corresponding field $K = \QQ_p$ has but five (unoriented) quadratic extensions, namely those obtained by adjoining a square root of $0$, $1$, $p$, $u$, and $pu$ where $u$ is an arbitrary non-square modulo $p$. The  quadratic ring extensions $S$ of $R$ then break up into five classes according to the corresponding extension $S_K$ of $K$. We will work out one representative case, namely the oriented ring extensions $S_n = \ZZ_p[p^n \sqrt{u}]$ corresponding to the unique unramified extension $L = K[\sqrt{u}]$ of degree $2$.

For any fractional ideal $I$ of $S_n$, we can pick an element of $I$ of minimal valuation (recalling that $L$ possesses a unique extension of the valuation on $K$) and scale it to be $1$. Then $S_n \subseteq I \subseteq S_0$, since $S_0 = \ZZ_p[\sqrt{u}]$ is the valuation ring, and it is easy to see that the only possible ideals are the subrings $S_0,S_1,\ldots, S_n$. In particular $S_n$ is the only invertible ideal class, and the class group $\Pic S$ is trivial.

We now enumerate the balanced triples that can be built out of these ideals. A balanced triple is formed from two sorts of data: three ideal classes $S_i$, $S_j$, $S_k$; and a scale factor $\gamma$ such that $\gamma S_i S_j S_k \subseteq S$ and
\[
  \<N(\gamma)\> = \frac{1}{\alpha(\Lambda^2 S_i) \alpha(\Lambda^2 S_j) \alpha(\Lambda^2 S_k)}.
\]
Computing
\[
  \alpha(\Lambda^2 S_i) = \alpha(1 \wedge p^i \sqrt{u}) = \<p^{i-n}\>,
\]
we get that $N(\gamma)$ has valuation $p^{3n-i-j-k}$ and in particular (since $L$ is unramified)
\begin{equation}\label{eq:padic mod 2}
  i + j + k \equiv n \mod 2.
\end{equation}
Write $3n-i-j-k = 2s$. Then $\gamma = p^{s} \gamma'$ where $\gamma' \in S_0^\cross$. To avoid needless repetition of arguments, we assume $i\leq j\leq k$, and then $\gamma S_i S_j S_k = p^s \gamma' S_i$. Let $\gamma' = a + b\sqrt{u}$ where $a,b \in \ZZ_p$. Since $p^s \gamma' S_i$ is clearly contained in $S_0$, the condition for it to lie in $S_n$ is that the irrational parts of its generators
\[
  p^s \gamma' \cdot 1 = p^s a + p^s b\sqrt{u} \textand
  p^s \gamma' \cdot p^i \sqrt{u} = p^{i+s}bu + p^{i+s}a\sqrt{u}
\]
are divisible by $p^n$, that is,
\[
  v_p(a) \geq n-s-i \textand v_p(b) \geq n-s.
\]
Since $a$ and $b$ cannot both be divisible by $p$, we must have $n-s-i\leq 0$, which can also be written as a sort of triangle inequality:
\begin{equation}\label{eq:tri ineq}
  (n-j) + (n-k) \geq n-i.
\end{equation}
If this holds, then the restrictions on $\gamma'$ are now merely that $p^{n-s} | b$, that is, $\gamma' \in S_t^\cross$ where $t = \max\{n-s, 0\}$. But if $\gamma'$ is multiplied by a unit in $S_i^\cross$, then the corresponding balanced triple is merely changed to an equivalent one. So the balanced triples are in bijection with the quotient $S_t^\cross/S_i^\cross$. Since the index of $S_i^\cross$ in $S_0^\cross$ is $p^{i-1}(p+1)$ $(i \geq 1)$, we have that there are precisely
\[
  B_{ijk} = \begin{cases}
    p^{i-t} & i \geq t > 0 \\
    p^{i-1} (p+1) & i > t = 0 \\
    1 & i=t=0
  \end{cases}
\]
classes of Bhargava boxes whose associated ideals are of the classes $S_i, S_j, S_k$, or equivalently, whose associated quadratic forms are
\[
  p^{n-i}x^2 -u p^{n+i}y^2, \quad p^{n-j}x^2 -u p^{n+j}y^2, \quad p^{n-k}x^2 -u p^{n+k}y^2.
\]

For beauty's sake let us examine one other angle of looking at the balanced triples. If we extend the notation $S_i$ ($i \in \ZZ$) to denote the $\ZZ_p$-module generated by $1$ and $p^i\sqrt{u}$ for every $i \in \ZZ$, then $S_i$ is an ideal of the ring $S_n$ exactly when $-n \leq i \leq n$. Of course $S_{-i} = p^{-i}\sqrt{u} \cdot S_i$ so we get no further ideal classes. But the admissible values of $i$, $j$, and $k$ now range in the stella octangula (Figure \ref{fig:stella}) formed by reflecting the graph of \eqref{eq:tri ineq} over the three coordinate planes, as well as the diagonal planes $i=j$, $i=k$, $j=k$. Indeed, the triples $(i,j,k)$ such that some scaling of $(S_i,S_j,S_k)$ is balanced are exactly the points of the lattice defined by \eqref{eq:padic mod 2} lying within the stella octangula. In such a case, one such balanced triple can be given by
\[
  (S_i,S_j,p^{s}S_k) \quad \text{or} \quad (S_i,S_j,p^{s}\sqrt{u} S_k)
\]
according as $(i,j,k)$ belongs to one or the other of the two tetrahedra making up the stella octangula.
\end{examp}

\section{Cubic algebras}
\label{sec:cubic}

The second main division of our paper has as its goal the parametrization of quartic algebras. We begin with cubic algebras, for there the parametrization is relatively simple and will also furnish the desired ring structure on the cubic resolvents of our quartic rings. The parametrization was done by Delone and Faddeev for cubic domains over $\ZZ$, by Gan, Gross, and Savin for cubic rings over $\ZZ$, and by Deligne over an arbitrary scheme (\cite{WQuartic}, p.~1074 and the references therein). Here we simply state and prove the result over a Dedekind domain, taking advantage of the construction in \cite{B3}, section 3.9.

\begin{thm}[cf.~\cite{B2}, Theorem 1; \cite{WQuartic}, Theorem 2.1; \cite{Poonen}, Proposition 5.1 and the references therein] \label{thm:cubic}
Let $R$ be a Dedekind domain. There is a canonical bijection between cubic algebras over $R$ and pairs consisting of a rank-2 $R$-lattice $M$ and a cubic map $\phi : M \to \Lambda^2 M$.
\end{thm}
\begin{proof}
Given the cubic ring $C$, we let $M = C/R$ so $\aa = \Lambda^2 M \cong \Lambda^3 C$ is an ideal class. Consider the map $\tilde{\phi} : C \to \aa$ given by $x \mapsto 1 \wedge x \wedge x^2$. This is a cubic map, and if $x$ is translated by an element $a \in R$, the map does not change. Hence it descends to a cubic map $\phi : M \to \aa$. We will show that each possible $\phi$ corresponds to exactly one ring $C$.

Fix a decomposition $M = \aa_1 \txi_1 \oplus \aa_2 \txi_2$ of $M$ into ideals. Any $C$ can be written as $R \oplus M = R\cdot 1 \oplus \aa_1 \xi_1 \oplus \aa_2 \xi_2$ as an $R$-module, where the lifts $\xi_1$ and $\xi_2$ are unique up to adding elements of $\aa_1^{-1}$ and $\aa_2^{-1}$ respectively. Then the remaining structure of $C$ can be described by a multiplication table
\begin{align*}
  \xi_1^2 &= \ell + a\xi_1 + b\xi_2 \\
  \xi_1\xi_2 &= m + c\xi_1 + d\xi_2 \\
  \xi_2^2 &= n + e\xi_1 + f\xi_2.
\end{align*}
It should be remarked that this is not literally a multiplication table for $C$, but rather for the corresponding $K$-algebra $C_K = C \otimes_R K$, which does literally have $\{1, \xi_1, \xi_2\}$ as a $K$-basis. For $C$ to be closed under this multiplication, the coefficients must belong to appropriate ideals ($\ell \in \aa_1^{-2}$, $a \in \aa_1^{-1}$, etc.).

Note that the basis change $\xi_1 \mapsto \xi_1 + t_1$, $\xi_2 \mapsto \xi_2 + t_2$ ($t_i \in \aa_i^{-1}$) diminishes $c$ and $d$ by $t_2$ and $t_1$, respectively (as well as wreaking greater changes on the rest of the multiplication table). Hence there is a unique choice of the lifts $\xi_1$ and $\xi_2$ such that $c=d=0$.

We now examine the other piece of data that we are given, the cubic map $\phi$ describable in these coordinates as
\begin{align*}
  \lefteqn{ \phi(x \txi_1 + y \txi_2)} \\
  &= 1 \wedge (x\xi_1 + y\xi_2) \wedge (x\xi_1 + y\xi_2)^2 \\
  &= 1 \wedge (x\xi_1 + y\xi_2) \wedge ((\ell + a\xi_1 + b\xi_2)x^2 + mxy + (n + e\xi_1 + f\xi_2)y^2)) \\
  &= (bx^3 - ax^2y + fxy^2 - ey^3)(1 \wedge \xi_1 \wedge \xi_2).
\end{align*}
Thus, in our situation, specifying $\phi$ is equivalent to specifying the four coefficients $a$, $b$, $e$, and $f$. It therefore suffices to prove that, for each quadruple of values $a \in \aa_1^{-1}$, $b \in \aa_1^{-2}\aa_2$, $e \in \aa_1\aa_2^{-2}$, $f \in \aa_2^{-1}$, there is a unique choice of values $\ell$, $m$, $n$, completing the multiplication table. The only conditions on the multiplication table that we have not used are the associative laws $(\xi_1^2)\xi_2 = \xi_1(\xi_1\xi_2)$ and $\xi_1(\xi_2^2) = (\xi_1\xi_2)\xi_2$. Expanding out these equations reveals the unique solution $\ell = -bf$, $m = be$, $n = -ae$, which indeed belong to the correct ideals. So from the map $\phi$ we have constructed a unique cubic ring $C$.
\end{proof}

\begin{examp}
Here we briefly summarize the most important examples over $R = \ZZ$, where the cubic map $\phi : M \to \Lambda^2 M$ reduces to a binary cubic form $\phi : \ZZ^2 \to \ZZ$, up to the twisted action of the group $\GL_2 \ZZ$ by
\[
  \left(\begin{bmatrix}
    a & b \\ c & d
  \end{bmatrix}
  . \phi\right)(x,y) = \frac{1}{ad-bc} \cdot \phi(ax + cy, bx + dy).
\]
\begin{itemize}
  \item The trivial ring $\ZZ[\epsilon_1,\epsilon_2]/(\epsilon_1^2,\epsilon_1 \epsilon_2, \epsilon_2^2)$ corresponds to the zero form $0$.
  \item Rings which are not domains correspond to reducible forms (e.g.~$\ZZ\oplus \ZZ\oplus \ZZ$ corresponds to $xy(x+y)$), and rings which have nontrivial nilpotents correspond to forms with repeated roots.
  \item A monogenic ring $\ZZ[\xi]/(\xi^3 + a\xi^2 + b\xi + c)$ corresponds to a form $x^3 + a x^2 y + b x y^2 + c y^3$ with leading coefficient $1$. Accordingly a form which does not represent the value $1$ corresponds to a ring that is not monogenic; for instance, the form $5x^3 + 7y^3$ (which attains only values $\equiv 0, \pm 2$ mod $7$) corresponds to the subring $\ZZ[\sqrt[3]{5^2\cdot 7}, \sqrt[3]{5\cdot 7^2}]$ of the field $\QQ[\sqrt[3]{5^2 \cdot 7}] = \QQ[\sqrt[3]{5\cdot 7^2}]$, proving that this ring (which is easily checked to be the full ring of integers in this field) is not monogenic.
  \item If a form $\phi$ corresponds to a ring $C$, then the form $n \cdot \phi$ corresponds to the ring $\ZZ + nC$ whose generators are $n$ times as large. Hence the content $\ct(\phi) = \gcd(a,b,c,d)$ of a form $\phi(x,y) = ax^3 + bx^2y + cxy^2 + dy^3$ equals the \emph{content} of the corresponding ring $C$, which is defined as the largest integer $n$ such that $C \cong \ZZ + nC'$ for some cubic ring $C'$. The notion of content (which is also not hard to define for cubic algebras over general Dedekind domains) will reappear prominently in our discussion of quartic algebras (see section \ref{ss:ri2res}).
\end{itemize}
\end{examp}

\section{Quartic algebras}
\label{sec:quartic}

Our next task is to generalize Bhargava's parametrization of quartic rings with a cubic resolvent in \cite{B3}, and in particular to formalize the notion of a cubic resolvent. The concept was first developed as part of the theory of solving equations by radicals, in which it was noted that if $a$, $b$, $c$, and $d$ are the unknown roots of a quartic, then
\[
  ab+cd, \quad ac+bd, \textand ad+bc
\]
satisfy a cubic whose coefficients are explicit polynomials in those of the original quartic. Likewise, if $Q \supseteq \ZZ$ is a quartic ring embeddable in a number field, the useful resolvent map
\[
  x \mapsto (\sigma_1(x)\sigma_2(x) + \sigma_3(x)\sigma_4(x),
  \sigma_1(x)\sigma_3(x) + \sigma_2(x)\sigma_4(x), \sigma_1(x)\sigma_4(x) + \sigma_2(x)\sigma_3(x))
\]
lands in a cubic subring of $\CC \oplus \CC \oplus \CC$, where $\sigma_1,\ldots,\sigma_4$ are the four embeddings $Q \hookrightarrow \CC$. The question then arises of what the proper notion of a resolvent map is in case $Q$ is not a domain. In section 2.1 of \cite{B3}, Bhargava defines from scratch a workable notion of Galois closure of a ring, providing a rank-$24$ algebra in which the resolvent can be defined. Alternatively (section 3.9), Bhargava sketches a way of axiomatizing the salient properties of a resolvent map. It is the second method that we develop here.

\begin{defn}[cf.~\cite{WQuartic}, p.~1069]
Let $R$ be a Dedekind domain, and let $Q$ be a quartic algebra over $R$. A \emph{resolvent} for $Q$ consists of a rank-$2$ $R$-lattice $M$, an $R$-module morphism $\theta : \Lambda^2 M \to \Lambda^3 (Q/R)$, and a quadratic map $\phi : Q/R \to M$ such that there is an identity of biquadratic maps
\begin{equation}\label{eq:resolvent}
  x \wedge y \wedge xy = \theta(\phi(x) \wedge \phi(y))
\end{equation}
from $Q \cross Q$ to $\Lambda^3 (Q/R)$.

The resolvent $(M,\theta,\phi)$ is called \emph{minimal} if $\phi$ has full image $\phi(Q/R) = M$, that is, it is not really a map to any proper sublattice $M' \subseteq M$ (cf.~Definition \ref{defn:image}). The resolvent is called \emph{numerical} if $\theta$ is an isomorphism.
\end{defn}
Our minimal resolvent corresponds to the ring $R^{\mathrm{inv}}$ in Bhargava's treatment (\cite{B3}, p.~1337), while our numerical resolvents correspond to Bhargava's resolvent. The numerical resolvents are more suited to analytic applications, while the minimal resolvent has the advantage of being canonical (for nontrivial $Q$), as we prove below.
\begin{examp}\label{ex:1st res}
For the prototypical example of a resolvent, take $Q = R^{\oplus 4}$ and $C = R^{\oplus 3}$, and let $M = C/R$. Let $\theta$ identify the standard orientations on these lattices, and let $\phi$ be given by the roots
\[
  \phi(a,b,c,d) = (ab + cd, ac + bd, ad + bc)
\]
of the classical resolvent of the quartic $(x-a)(x-b)(x-c)(x-d)$. It is easy to check that this is a resolvent, which is both minimal and numerical. Many more examples can be derived from this (see Example \ref{ex:res z}).
\end{examp}

\subsection{Resolvent to ring}

Our first result is that the resolvent encapsulates the data of the ring:

\begin{thm}[cf.~\cite{B3}, Theorem 1 and Proposition 10; \cite{WQuartic}, Corollary 1.2] \label{thm:quartic}
  Let $\tilde Q$ and $M$ be $R$-lattices of ranks $3$ and $2$ respectively. Let $\theta : \Lambda^2 M \to \Lambda^3 \tilde Q$ be a morphism, and let $\phi : \tilde Q \to M$ be a quadratic map. Then there is a unique quartic ring $Q$ with an isomorphism $Q/R \cong \tilde{Q}$ such that $(M,\theta,\phi)$ is a resolvent for $Q$.
\end{thm}

\begin{proof}
Write $\tilde Q = \aa_1\txi_1 \oplus \aa_2\txi_2 \oplus \aa_3\txi_3$ as usual. The ring $Q$ will of course be $R \oplus \aa_1\xi_1 \oplus \aa_2\xi_2 \oplus \aa_3\xi_3$ as an $R$-module, with a multiplication table
\[
  \xi_i\xi_j = c_{ij}^0 + \sum_{k=1}^3 c_{ij}^k \xi_k
\]
where $c_{ij}^0 \in \aa_i^{-1}\aa_j^{-1}$ and $c_{ij}^k \in \aa_i^{-1}\aa_j^{-1}\aa_k$. The $18$ coefficients $c_{ij}^k$ are subject to the expansion of the relation \eqref{eq:resolvent}:
\begin{equation}\label{eq:xpresolvent}
  \(\sum_i x_i\txi_i\) \wedge \(\sum_j y_j\txi_j\) \wedge \(\sum_{i,j,k} x_i y_j c_{ij}^k \txi_k\) = \theta\(\phi\(\sum_i x_i\txi_i\) \wedge \phi\(\sum_j y_j\txi_j\)\).
\end{equation}
Write
\[
  \phi(x_1\xi_1 + x_2\xi_2 + x_3\xi_3) = \sum_{1\leq i \leq j \leq 3} \mu_{ij} x_i x_j
\]
where $\mu_{ij} \in \aa_i^{-1}\aa_j^{-1}M$. Then define
\[
  \lambda^{ij}_{k\ell} = \theta(\mu_{ij} \wedge \mu_{k\ell}) \in \aa_1\aa_2\aa_3\aa_i^{-1}\aa_j^{-1}\aa_k^{-1}\aa_\ell^{-1}.
\]
We can now expand both sides of \eqref{eq:xpresolvent} as polynomials in the $x$'s and $y$'s times $\txi_1 \wedge \txi_2 \wedge \txi_3$, getting
\[
  \begin{vmatrix}
    x_1 & y_1 & \sum_{i,j} c_{ij}^1 x_i y_j \\
    x_2 & y_2 & \sum_{i,j} c_{ij}^2 x_i y_j \\
    x_3 & y_3 & \sum_{i,j} c_{ij}^3 x_i y_j
  \end{vmatrix}
  = \sum_{i\leq j} \sum_{k\leq \ell} \lambda^{ij}_{k\ell} x_i x_j y_k y_\ell,
\]
and equate coefficients of each biquadratic monomial $x_i x_j y_k y_\ell$. Due to the skew-symmetry of each side, all terms involving $x_i^2 y_i^2$ or $x_ix_jy_iy_j$ cancel, and the remaining $30$ equations group into $15$ matched pairs. They are summarized as follows, where $(i,j,k)$ denotes any permutation of $(1,2,3)$ and $\epsilon = \pm 1$ its sign:
\begin{equation}\label{eq:c-lam}
\begin{aligned}
  c_{ii}^j &= -\epsilon \lambda^{ii}_{ik} \\
  c_{ij}^k &= \epsilon \lambda^{jj}_{ii} \\
  c_{ij}^j - c_{ik}^k &= \epsilon \lambda^{jk}_{ii} \\
  c_{ii}^i - c_{ij}^j - c_{ik}^k &= \epsilon \lambda^{ij}_{ik}.
\end{aligned}
\end{equation}
At first glance it may seem that one can add a constant $a$ to $c_{ij}^j$ and $c_{ij}^k$, while adding $2a$ to $c_{ii}^i$, to derive a three-parameter family of solutions from a single one; but this is merely the transformation induced by the change of lift $\xi_i \mapsto \xi_i + a$ for $\txi_i$. So there is essentially only one solution. (It could be normalized by taking e.g{.} $c_{12}^1 = c_{23}^2 = c_{31}^3 = 0$, although we do not use this normalization here, preferring to save time later by keeping the indices $1$, $2$, and $3$ in complete symmetry.)

The constant terms $c_{ij}^0$ of the multiplication table are as yet undetermined. They must be deduced from the associative law. There are several ways to compute each $c_{ij}^0$, and to prove that they agree, along with all the other relations implied by the associative law, is the final step in the construction of the quartic ring $Q$. Our key tool is the \emph{Pl\"ucker relation} relating the wedge products of four vectors in a $2$-dimensional space:
\[
  (a \wedge b)(c \wedge d) + (a \wedge c)(d \wedge b) + (a \wedge d)(b \wedge c) = 0,
\]
or, as we will use it,
\[
  \lambda^{aa'}_{bb'} \lambda^{cc'}_{dd'} + \lambda^{aa'}_{cc'} \lambda^{dd'}_{bb'} + \lambda^{aa'}_{dd'} \lambda^{bb'}_{cc'} = 0.
\]
To give succinct names to these relations among the $\lambda$'s, note that $aa',\ldots, dd'$ are four of the six unordered pairs that can be formed from the symbols $1$, $2$, and $3$, and the relation is nontrivial only when these four pairs are distinct. Consequently we denote it by $\Pl(ee',f\!f')$, where $ee'$ and $f\!f'$ are the two pairs that do not appear in it. Then $\Pl(ee',f\!f')$ as a polynomial in the $\lambda$'s is unique up to sign, and we will never have occasion to fix a sign convention.

We are now ready to derive the associative law from the Pl\"ucker relations. Of course this is a task that could be left to a computer, but since we will soon be deriving the Pl\"ucker relations from the associative law, we find it advisable to present the process at least in summary form. Here it is:
\begin{equation}\label{diag:plu}
\begin{gathered}{}
  [(\xi_i\xi_i)\xi_j - (\xi_i\xi_j)\xi_i]_k = \Pl(jk,kk) \\
  [(\xi_i\xi_j)\xi_k - (\xi_i\xi_k)\xi_j]_i = \Pl(ij,ik) \\
  \xymatrix{
    [(\xi_i\xi_i)\xi_j - (\xi_i\xi_j)\xi_i]_j \ar@{-}[d]^{\Pl(jj,kk)} &
    [(\xi_i\xi_j)\xi_i - (\xi_i\xi_i)\xi_j]_i \ar@{-}[r]^{\Pl(ij,kk)} \ar@{-}[d]^{\Pl(ik,jk)} &
    [(\xi_i\xi_j)\xi_k - (\xi_j\xi_k)\xi_i]_k \ar@{-}[d]^{\Pl(ik,jk)} \\
    [(\xi_i\xi_i)\xi_k - (\xi_i\xi_k)\xi_i]_j &
    [(\xi_i\xi_j)\xi_j - (\xi_j\xi_j)\xi_i]_j \ar@{-}[r]^{\Pl(ij,kk)} &
    [(\xi_i\xi_j)\xi_k - (\xi_j\xi_k)\xi_i]_k
  }
\end{gathered}
\end{equation}
And here is the explanation:
\begin{itemize}
\item The notation $[\omega]_i$ denotes the coefficient of $\xi_i$ when $\omega$ is expressed in terms of the basis $\{1,\xi_1,\xi_2,\xi_3\}$.
\item Each of the first two equations is a direct calculation. For instance:
\begin{align*}
  \lefteqn{[(\xi_i\xi_i)\xi_j - (\xi_i\xi_j)\xi_i]_k} \\
  &= [(c_{ii}^0 + c_{ii}^i\xi_i + c_{ii}^j\xi_j + c_{ii}^k\xi_k)\xi_j
     -(c_{ij}^0 + c_{ij}^i\xi_i + c_{ij}^j\xi_j + c_{ij}^k\xi_k)\xi_i]_k \\
  &= c_{ii}^i c_{ij}^k + c_{ii}^j c_{jj}^k + c_{ii}^k c_{jk}^k - c_{ij}^i c_{ii}^k - c_{ij}^j c_{ij}^k - c_{ij}^k c_{ik}^k \\
  &= (c_{ii}^i - c_{ij}^j - c_{ik}^k)c_{ij}^k + c_{ii}^k(c_{jk}^k - c_{ij}^i) + c_{ii}^j c_{jj}^k \\
  &= \epsilon(\lambda^{ij}_{ik} \lambda^{jj}_{ii} - \lambda^{ii}_{ij}\lambda^{ik}_{jj} + \lambda^{ii}_{ik} \lambda^{jj}_{ij}) \\
  &= \Pl(jk,kk).
\end{align*}
\item The two lower diagrams show the instances of the associative law that produce a summand of $c_{ii}^0$ or $c_{ij}^0$, respectively. Each node in the diagrams yields a formula for $c_{ii}^0$ or $c_{ij}^0$ (having no denominator, and consequently belonging to the correct ideal $\aa_i^{-2}$ resp.~$\aa_i^{-1}\aa_j^{-1}$); and where two nodes are joined by a line, the \emph{difference} between the two corresponding formulas is expressible as a Pl\"ucker relation.
\end{itemize}
We have now proved all of the associative law except the constant terms; that is, we now have that $(xy)z - x(yz) \in R$ for all $x,y,z \in Q$. Attacking the constant terms in the same manner as above leads to considerably heavier computations, which could be performed by computer (compare \cite{B3}, top of p.~1343). Alternatively, we may use the following trick. Let $i,j,k\in \{1,2,3\}$ be any indices, and let $h \in \{1,2,3\}$ be an index distinct from $k$. Then using the already-proved $\xi_h$-component of the associative law,
\begin{align*}
  \xi_i(\xi_j\xi_k) - \xi_j(\xi_i\xi_k)
  &= [\xi_h(\xi_i(\xi_j\xi_k)) - \xi_h(\xi_j(\xi_i\xi_k))]_h \\
  &= [(\xi_h\xi_i)(\xi_j\xi_k) - (\xi_h\xi_j)(\xi_i\xi_k)]_h \\
  &= [((\xi_h\xi_i)\xi_j)\xi_k - ((\xi_h\xi_j)\xi_i)\xi_k]_h.
\end{align*}
This last is necessarily zero, since it consists of the number $(\xi_h\xi_i)\xi_j - (\xi_h\xi_j)\xi_i \in R$ multiplied by $\xi_k$, and thus has no $\xi_h$-component.
\end{proof}

\subsection{Ring to resolvent}
\label{ss:ri2res}

Conversely, we will now study all possible resolvents of a given quartic ring $Q$. There is one case in which this problem takes a striking turn: the \emph{trivial} ring $Q = R[\aa_1\epsilon_1,\aa_2\epsilon_2,\aa_3\epsilon_3]/\sum_{i,j}(\aa_i\aa_j\epsilon_i\epsilon_j)$ where all entries of the multiplication table are zero. Here $\phi$ can be an arbitrary map to a $1$-dimensional sublattice of $M$, or alternatively $M$ and $\phi$ can be chosen arbitrarily while $\theta = 0$. For all other quartic rings, the family of resolvents is much smaller, as we will now prove.
\begin{thm}[cf.~\cite{B3}, Corollary 18] \label{thm:content}
Let $Q$ be a nontrivial quartic $R$-algebra. Then
\begin{enumerate}[$(a)$]
\item $Q$ has a unique minimal resolvent $(M_0, \theta_0, \phi_0)$;
\item we have $\theta_0(\Lambda^2 M_0) = \cc \cdot \Lambda^3 (Q/R)$, where $\cc$ is the ideal (called the \emph{content} of $Q$) characterized by the following property: For each ideal $\aa \subseteq R$, there exists a quartic $R$-algebra $Q'$ such that $Q \cong R + \aa Q'$ if and only if $\aa | \cc$;
\item all other resolvents $(M,\theta,\phi)$, up to isomorphism, are found by extending $\theta_0$ linearly to $\Lambda^2 M$, where $M$ is a lattice with $[M:M_0]\mid\cc$, and taking $\phi = \phi_0$;
\item the numerical resolvents arise by taking $[M:M_0]=\cc$ in the preceding.
\end{enumerate}
\end{thm}
\begin{proof}
Write $Q = R \oplus \aa_1\xi_1 \oplus \aa_2\xi_2 \oplus \aa_3\xi_3$. The multiplication table can be encoded in a family of $c_{ij}^k$'s, from which the fifteen values $\lambda^{ij}_{k\ell}$ are determined through \eqref{eq:c-lam}. These $\lambda^{ij}_{k\ell}$ satisfy the fifteen Pl\"ucker relations by \eqref{diag:plu}. It then remains to construct the target module $M$, the map $\theta$, and the vectors $\mu_{ij} \in \aa_i^{-1}\aa_j^{-1}M$ such that their pairwise exterior products $\mu_{ij} \wedge \mu_{k\ell}$ have, via $\theta$, the specified value $\lambda^{ij}_{k\ell}$.

The six $\mu_{ij}$ are in complete symmetry at this point, and it will be convenient to denote $\mu_{ij}$ by $\mu_x$, where $x$ runs over $\{11,12,13,22,23,33\}$ or, if you prefer, $\{1,2,3,4,5,6\}$. Likewise we write each $\lambda^{ij}_{k\ell}$ as $\lambda^x_y$ or simply $\lambda_{xy}$.

We first tackle the problem over $K$. Let $V$ be an abstract $K$-vector space of dimension $2$. We construct vectors $\mu_1,\ldots, \mu_n$ whose exterior products are proportional to the $\lambda$'s as follows. Some $\lambda_{xy}$ is nonzero, without loss of generality $\lambda_{12}$. Let $(\mu_1,\mu_2)$ be a basis of $V$. Then, for $3 \leq x \leq 6$, take
\[
  \mu_x = \frac{-\lambda_{2x}\mu_1 + \lambda_{1x}\mu_2}{\lambda_{12}}
\]
to give the products $\mu_1 \wedge \mu_x$ and $\mu_2 \wedge \mu_x$ the desired values. The $\lambda_{xy}$ with $3 \leq x < y \leq 6$ have not been used, but their values were forced by the Pl\"ucker relations anyway, so we have a system of $\mu_x$ such that
\[
  \mu_x \wedge \mu_y = \frac{\lambda_{xy}}{\lambda_{12}} \cdot \mu_1 \wedge \mu_2.
\]
Moreover, these are the \emph{only} $\mu_x \in V$ with this property, up to $\GL(V)$-trans\-formations.

Now define a quadratic map $\phi_0 : Q/R \to V$ by
\[
  \phi_0(x_1\xi_1 + x_2\xi_2 + x_3\xi_3) = \sum_{i \leq j} \mu_{ij} x_i x_j
\]
and a linear map $\theta_0 : \Lambda^2 V \to \Lambda^3(Q/R)\tensor_R K$ by
\[
  \theta_0(\mu_1\wedge \mu_2) = \lambda_{12} (\xi_1 \wedge \xi_2 \wedge \xi_3).
\]
We have that $(V,\theta_0,\phi_0 \tensor K)$ is the unique resolvent of the quartic $K$-algebra $Q \tensor_R K$. Resolvents for $Q$ are now in bijection with lattices $M \subseteq V$ such that
\begin{equation}\label{eq:allres}
  M \supseteq \phi_0(Q/R) \textand \theta_0(\Lambda^2 M) \subseteq \Lambda^3(Q/R).
\end{equation}
There is now clearly at most one minimal resolvent, gotten by taking $M$ to be the image $M_0 = \phi(Q/R)$. We have
\begin{align*}
  \theta_0(\Lambda^2 M_0) &= \theta_0\left(\sum_{i,j,k,\ell} \aa_i \aa_j \aa_k \aa_\ell \cdot \mu_{ij} \wedge \mu_{k\ell}\right) \\
  &= \left(\sum_{i,j,k,\ell} \lambda^{ij}_{k\ell} \aa_i \aa_j \aa_k \aa_\ell\right) \xi_1 \wedge \xi_2 \wedge \xi_3 = \cc \Lambda^3(Q/R),
\end{align*}
where
\[
  \cc = \sum_{i,j,k,\ell} \lambda^{ij}_{k\ell} \aa_i \aa_j \aa_k \aa_\ell \aa_1^{-1} \aa_2^{-1} \aa_3^{-1}.
\]
The ideal in which $\lambda^{ij}_{k\ell}$ is constrained to live is $\aa_1\aa_2\aa_3\aa_i^{-1}\aa_j^{-1}\aa_k^{-1}\aa_\ell^{-1}$; so $\cc \subseteq R$ and there is a unique minimal resolvent, proving (a).

If $\aa \supseteq \cc$ for some $\aa \subseteq R$, we can replace each of the three $\aa_i$ with $\aa^{-1}\aa_i$ without changing the validity of the $\lambda$-system. This means there is an extension ring $Q' = R \oplus \aa^{-1} \aa_1\xi_1 \oplus \aa^{-1} \aa_2\xi_2 \oplus \aa^{-1} \aa_3\xi_3$ with the same multiplication table as $Q$, and we see that $Q = R + \aa Q'$. Conversely, given such a $Q'$, we write its multiplication table with respect to the basis $Q'/R = \aa^{-1} \aa_1\xi_1 \oplus \aa^{-1} \aa_2\xi_2 \oplus \aa^{-1} \aa_3\xi_3$ and get that $\lambda^{ij}_{k\ell} \in \aa\aa_1\aa_2\aa_3\aa_i^{-1}\aa_j^{-1}\aa_k^{-1}\aa_\ell^{-1}$, so $\cc \subseteq \aa$. This proves (b).

Finally, the relation $\theta_0(\Lambda^2 M_0) = \cc \Lambda^3(Q/R)$ allows us to rewrite \eqref{eq:allres} as
\[
  M \supseteq M_0 \textand \Lambda^2 M \subseteq \cc \Lambda^2 M_0.
\]
Now (c) is obvious. A numerical resolvent occurs when $\theta_0(\Lambda^2 M) = \Lambda^3(Q/R)$, so (d) is obvious as well.
\end{proof}

Bhargava proved (\cite{B3}, Corollary 4) that the number of (numerical) resolvents of a quartic ring over $\ZZ$ is the sum of the divisors of its content. Likewise, we now have:
\begin{cor}
If $\cc \neq 0$, then the numerical resolvents of $Q$ are in noncanonical bijection with the disjoint union
\[
  \coprod_{R\supseteq \aa \supseteq \cc} R/\aa.
\]
\end{cor}
\begin{proof}
Here we simply have to count the superlattices $M$ of index $\cc$ over a fixed lattice $M_0$. The classical argument over $\ZZ$ extends rather readily; for completeness, we give the proof.

Note that we must have $M \subseteq \cc^{-1}M_0$, since $M \wedge M_0 \subseteq M \wedge M = \cc^{-1} \Lambda^2 M_0$. Pick a decomposition $\cc^{-1}M_0 \cong \dd_1 \oplus \dd_2$. Then consider the map $\pi : M \to \dd_1$ that is the restriction of projection to the first factor. We have $\ker \pi = \{0\} \times \aa \dd_2$ and $\im \pi = \bb \dd_1$ for some ideals $\aa,\bb$ subject to the familiar behavior of top exterior powers in exact sequences:
\[
  \cc^{-1}\Lambda^2 M_0 = \Lambda^2 M = \aa \dd_2 \wedge \bb \dd_1 = \aa\bb \cc^{-2} \Lambda^2 M_0,
\]
that is, $\aa\bb = \cc$. Now if $\aa$ and $\bb$ are fixed, the lattice $M$ is determined by a picking a coset in $\dd_2/\aa \dd_2$ to be the preimage of each point $b \in \im \pi$; this is determined by an $R$-module map
\[
  \bb \dd_1 \to \dd_2/\aa \dd_2
\]
or, since $\cc \dd_1$ is necessarily in the kernel,
\[
  \bb \dd_1/\cc \dd_1 \to \dd_2/\aa \dd_2.
\]
We can identify both the domain and the target of this map with $R/\aa$ via the standard result that if $\aa$ and $\bb$ are ideals in a Dedekind domain $R$, then $\aa/\aa\bb \cong R/\bb$. (Proof: Use the Chinese Remainder Theorem to find $a \in \aa$ that has minimal valuation with respect to each of the primes dividing $\bb$. Then $a$ generates $\aa/\aa\bb$, and $a \mapsto 1$ is the desired isomorphism.) Then the desired parameter space is $\Hom_R(R/\aa,R/\aa) \cong R/\aa$. Letting $\aa$ vary yields the claimed bijection.
\end{proof}

In particular, we have the following.
\begin{cor}[cf.~\cite{B3}, Corollary 5] \label{cor:at least 1 res}
Every quartic algebra over a Dedekind domain possesses at least one numerical resolvent.
\end{cor}

\subsection{The cubic ring structure of the resolvent}

In contrast to the classical presentation, the resolvent maps we have constructed take their values in \emph{modules}, without any explicit connection to a cubic ring. In fact, there is the structure of a cubic ring already latent in a resolvent:

\begin{thm} \label{thm:cub ring struc}
To any quartic ring $Q$ and resolvent $(M,\theta,\phi)$ thereof, one can canonically associate a cubic ring $C$ with an identification $C/R \cong M$.
\end{thm}
\begin{rem}
As stated, this theorem has no content, as one can take the trivial ring structure on $R \oplus M$. However, we will produce a ring structure generalizing the classical notion of cubic resolvent. This $C$ may be called a cubic resolvent of $Q$, the maps $\theta$ and $\phi$ being suppressed.
\end{rem}

\begin{proof}
We use the following trick of multilinear algebra (compare \cite{WQuartic}, p.~1076). First pick a decomposition $Q/R = \aa_1\txi_1 \oplus \aa_2\txi_2 \oplus \aa_3\txi_3$. Writing
\[
  \phi(x_1\txi_1 + x_2\txi_2 + x_3\txi_3) = \sum_{i\leq j} x_i x_j \mu_{ij} \quad
  (\mu_{ij} \in \aa_i^{-1} \aa_j^{-1} M),
\]
consider the determinant
\begin{align*}
  \Delta &= 4 \det \begin{bmatrix}
    \mu_{11} & \frac{1}{2} \mu_{12} & \frac{1}{2} \mu_{13} \\
    \frac{1}{2} \mu_{12} & \mu_{22} & \frac{1}{2} \mu_{23} \\
    \frac{1}{2} \mu_{13} & \frac{1}{2} \mu_{23} & \mu_{33}
  \end{bmatrix} \\
  &= 4\mu_{11}\mu_{22}\mu_{33} + \mu_{12}\mu_{13}\mu_{23} - \mu_{11}\mu_{23}^2 - \mu_{22}\mu_{13}^2 - \mu_{33}\mu_{12}^2 \quad \\
  &\in \aa_1^{-2}\aa_2^{-2}\aa_3^{-2} \Sym^3 M
\end{align*}
(the two expressions are equal except when $\cha K = 2$, in which case the first becomes purely motivational). Next, $\theta$ allows us to map $\aa_1^{-2}\aa_2^{-2}\aa_3^{-2}$ to $(\Lambda^2 M)^{\tensor -2}$. The $\Lambda^2 M$-valued pairing $\wedge$ on $M$ gives an identification of $M$ with $\Lambda^2 M \tensor M^*$, so we can transform
\begin{align*}
  (\Lambda^2M)^{\tensor -2} \tensor \Sym^3 M &\cong (\Lambda^2M)^{\tensor-2} \tensor \Sym^3((\Lambda^2 M) \tensor M^*) \\
  &\cong (\Lambda^2M)^{\tensor-2} \tensor (\Lambda^2 M)^{\tensor 3} \tensor \Sym^3(M^*) \\
  &\cong (\Lambda^2M) \tensor \Sym^3(M^*).
\end{align*}
Thus $\Delta$ yields a cubic map $\delta : M \to \Lambda^2 M$, which by Theorem \ref{thm:cubic} is equivalent to a cubic ring $C$ with an identification $C/R \cong M$. That $\delta$ is independent of the chosen basis $(\txi_1, \txi_2, \txi_3)$ is a polynomial identity that follows from properties of the determinant, at least when $\cha K \neq 2$.
\end{proof}

Two theorems concerning this cubic ring structure we will state without proof, since they are mere polynomial identities already implied by Bhargava's work over $\ZZ$. The first may be used as an alternative to Theorem \ref{thm:cubic} to determine the multiplicative structure on $C$; as Bhargava notes, it uniquely determines the ring $C$ in all cases over $\ZZ$ except when $Q$ has nilpotents.

\begin{thm}[cf.~\cite{B3}, equation (30)]
Let $Q$ be a quartic ring, and let $C$ be the cubic ring whose structure is determined by the resolvent map data $\theta : \Lambda^2(C/R) \to \Lambda^3(Q/R)$ and $\phi : Q/R \to C/R$. For any element $x \in Q$ and any lift $y \in C$ of the element $\phi(x) \in C/R$, we have the equality
\[
  x \wedge x^2 \wedge x^3 = \theta(y \wedge y^2).
\]
\end{thm}

We end this section with a theorem concerning discriminants, which until now have been conspicuously absent from our discussion, in direct contrast to Bhargava's presentation. Recall that the discriminant of a $\ZZ$-algebra $Q$ with a $\ZZ$-basis $(\xi_1,\ldots,\xi_n)$ is defined as the determinant of the matrix $[\tr(\xi_i\xi_j)]_{i,j}$. In like manner, define the \emph{discriminant} of a rank-$n$ $R$-algebra $Q$ to be the map
\[
  \disc(Q) : x_1 \wedge \cdots \wedge x_n \mapsto \det[\tr(x_i x_j)]_{i,j}.
\]
It is quadratic and thus can be viewed as an element of $(\Lambda^n Q^*)^{\otimes 2}$, a rank-$1$ lattice that is not in general isomorphic to $R$. The discriminants of a quartic ring and its resolvents are ``equal'' in precisely the way one might hope:

\begin{thm}[cf.~\cite{B3}, Proposition~13]
Let $Q$, $C$, $\theta$ be as above. The morphism
\[
  (\theta^*)^{\otimes 2} : (\Lambda^3(Q/R)^*)^{\otimes 2} \to (\Lambda^2(C/R)^*)^{\otimes 2}
\]
carries $\disc Q$ to $\disc C$.
\end{thm}

\begin{examp}\label{ex:res z}
Once again, we recapitulate the situation over $\ZZ$. Here, once bases $Q/R = \ZZ\xi_1 \oplus \ZZ\xi_2 \oplus \ZZ\xi_3$ and $C/R = \ZZ\eta_1 \oplus \ZZ\eta_2$ have been fixed so that $\theta$ is given simply by $\eta_1 \wedge \eta_2 \mapsto \xi_1 \wedge \xi_2 \wedge \xi_3$, the remaining datum $\phi$ of a numerical resolvent can be written as a pair of ternary quadratic forms, or, even more pictorially, as a pair of symmetric matrices
\[
  (A,B) = 
  \left(
  \begin{bmatrix}
    a_{11} & \frac{1}{2} a_{12} & \frac{1}{2} a_{13} \\
    \frac{1}{2} a_{12} & a_{22} & \frac{1}{2} a_{23} \\
    \frac{1}{2} a_{13} & \frac{1}{2} a_{23} & a_{33}
  \end{bmatrix},
  \begin{bmatrix}
    b_{11} & \frac{1}{2} b_{12} & \frac{1}{2} b_{13} \\
    \frac{1}{2} b_{12} & b_{22} & \frac{1}{2} b_{23} \\
    \frac{1}{2} b_{13} & \frac{1}{2} b_{23} & b_{33}
  \end{bmatrix}
  \right).
\]
where $a_{ij}, b_{ij} \in \ZZ$. The associated cubic ring is found by applying Theorem \ref{thm:cubic} to the form $4 \det(Ax + By)$. Some salient examples follow:
\begin{itemize}
\item First note that there is a resolvent map of $\CC$-algebras from $Q_0 = \CC^{\oplus 4}$ to $C_0 = \CC^{\oplus 3}$ given by the roots of the equation-solver's resolvent
\[
  (x,y,z,w) \mapsto (xy + zw, xz + yw, xw + yz)
\]
or, more accurately, by its reduction modulo $\CC$
\[
  \phi_0 : (x,y,z,0) \mapsto (xy - yz, xz - yz, 0),
\]
supplemented of course by the standard identification
\[
  \theta_0 : \Lambda^2 (C_0/\CC) \to \Lambda^3 (Q_0/\CC).
\]

Accordingly, if we have a quartic $\ZZ$-algebra $Q \subseteq Q_0$ and a cubic $\ZZ$-algebra $C \subseteq C_0$ on which the restrictions of $\phi_0$ and $\theta_0$ are well-defined, then it automatically follows that $C/\ZZ$ is a resolvent for $Q$ with attached cubic ring structure $C$.
\item As an example, consider the ring
\[
  Q = \ZZ + p(\ZZ \oplus \ZZ \oplus \ZZ \oplus \ZZ) = \{(a,b,c,d) \in \ZZ^{\oplus 4} : a\equiv b\equiv c\equiv d \bmod p\}
\]
of content $p$, for each prime $p$. The minimal resolvent of $Q$ comes out to be $\phi_0(Q/\ZZ) = C'/\ZZ$, where
\[
  C' = \ZZ + p^2 \cdot \ZZ^{\oplus 3}.
\]
But $C'$ is not a numerical resolvent of $Q$: it has index $p^4$ in $\ZZ^{\oplus 3}$, while $Q$ has index $p^3$ in $\ZZ^{\oplus 4}$, so the restriction of $\theta_0$ cannot possibly be an isomorphism. We must enlarge $C'$ by a factor of $p$. Note that any subgroup $C$ such that
\[
  \ZZ + p^2 \cdot \ZZ^{\oplus 3} \subseteq C \subseteq \ZZ + p \cdot \ZZ^{\oplus 3}
\]
is a ring, since the product of two elements in $p \cdot \ZZ^{\oplus 3}$ lies in $p^2 \cdot \ZZ^{\oplus 3}$.
So any ring of the form
\[
  C = \ZZ + p^2 \cdot \ZZ^{\oplus 3} + \<ap,bp,0\>
\]
is a numerical resolvent of $Q$. Letting $[a:b]$ run over $\PP^1(\ZZ/p\ZZ)$ yields the $p+1$ numerical resolvents predicted by Theorem \ref{thm:content}.
\item Note that some of these resolvents are isomorphic under the automorphism group of $Q$, which is simply $S_4$ acting by permuting the coordinates. One verifies that $S_4$ acts through its quotient $S_3$, which in turn permutes the three distinguished points $0,1,\infty$ on $\PP^1(\ZZ/p\ZZ)$. Accordingly, if we are using Theorem \ref{thm:quartic} to count quartic rings, the ring $Q$ will appear not $p+1$ times but $\lceil p/6 \rceil + 1$ times ($1$ time if $p = 2$). This is no contradiction with Theorem \ref{thm:content}, which gives the number of resolvents \emph{as maps out of the given ring $Q$.}
\end{itemize}
\end{examp}

\section{Maximality}
\label{sec:maximality}
In order to convert his parametrization of quartic rings into one of quartic fields, Bhargava needed a condition for a ring to be maximal, i.e.~to be the full ring of integers in a field. In like manner we discuss how to tell if a quartic ring $Q$ over a Dedekind domain $R$ is maximal in its fraction ring $Q_K$ using conditions on a numerical resolvent. The first statement to make is that maximality is a local condition, i.e.~can be checked at each prime ideal.

\begin{prop}
Let $Q$ be a ring of finite rank $n$ over $R$. $Q$ is maximal if and only if $Q_\pp = Q \otimes_R R_\pp$ is maximal over $R_\pp$ for all (nonzero) primes $\pp \subseteq R$.
\end{prop}
\begin{rem}
Here, $R_\pp$ denotes the localization
\[
  R_\pp = \left\{ \frac{a}{b} \in K : a \in R, b \in R \setminus \pp \right\}
\]
(not its completion as with the symbol $\ZZ_p$).
\end{rem}
\begin{proof}
When $Q$ is a domain, one can use the facts that $Q$ is maximal if and only if it is \emph{normal} (integrally closed in its fraction field) and that normality is a local property. A direct proof is also not difficult.

Suppose that $Q$ is not maximal, so that there is a larger ring $Q'$ with $Q_K = Q'_K$. The nonzero $R$-module $Q'/Q$ is pure torsion, so there is a prime $\pp$ such that $(Q'/Q)_\pp = Q'_\pp/Q_\pp$ is nonzero, i.e.~$Q_\pp$ embeds into the larger ring $Q'_\pp$.

Suppose now that for some $\pp$, $Q_\pp$ embeds into a larger ring $T$. We construct an extension ring $Q'$ of $Q$ by the formula
\[
  Q' = Q[\pp^{-1}] \intsec T.
\]
It is obvious that $Q'$ is a ring containing $Q$; it is not so obvious that it is a rank-$n$ ring, in other words, that it is finitely generated as an $R$-module. Let $X$ be the $R$-lattice generated by any $K$-basis $x_1,\ldots,x_n$ of $Q_K$. Since $Q$ and $T$ are finitely generated $R$- and $R_\pp$-modules respectively, we may divide the $x_i$ by sufficiently divisible elements of $R$ to assume $Q \subseteq X$ and $T \subseteq R_\pp X$. Then
\[
  Q' \subseteq X[\pp^{-1}] \intsec R_\pp X.
\]
Note that
\begin{align*}
  R_\pp X &= \left\{\sum_i a_i x_i : v_\pp(a_i) \geq 0 \right\} \\
  X[\pp^{-1}] &= \left\{\sum_i a_i x_i : v_\qq(a_i) \geq 0 \quad \forall\qq \neq \pp\right\} \\
  X[\pp^{-1}] \intsec R_\pp X &= \left\{\sum_i a_i x_i : v_\qq(a_i) \geq 0 \quad \forall\qq\right\} = X, \\
\end{align*}
whence $Q' \subseteq X$ is finitely generated.

Finally we must show that $Q' \neq Q$. This is obvious by localization:
\[
  Q'_\pp = (Q[\pp^{-1}])_\pp \intsec T_\pp = Q_K \intsec T = T \neq Q_\pp. \qedhere
\]
\end{proof}
The local rings $R_\pp$ are DVR's, and in particular are PID's, so we can visualize a localized numerical resolvent $(Q_\pp, M_\pp, \theta, \phi)$ in a simple way: Pick bases $Q_\pp/R_\pp = {R_\pp}{\langle\txi_1,\txi_2,\txi_3\rangle}$ and $M_\pp = {R_\pp}{\<\eta_1,\eta_2\>}$ such that $\theta(\eta_1 \wedge \eta_2) = \txi_1 \wedge \txi_2 \wedge \txi_3$, and write $\phi$ as a pair of matrices
\[
  (A,B) = \left(
  \begin{bmatrix}
    a_{11} & \frac{1}{2} a_{12} & \frac{1}{2} a_{13} \\
    \frac{1}{2} a_{12} & a_{22} & \frac{1}{2} a_{23} \\
    \frac{1}{2} a_{13} & \frac{1}{2} a_{23} & a_{33}
  \end{bmatrix},
  \begin{bmatrix}
    b_{11} & \frac{1}{2} b_{12} & \frac{1}{2} b_{13} \\
    \frac{1}{2} b_{12} & b_{22} & \frac{1}{2} b_{23} \\
    \frac{1}{2} b_{13} & \frac{1}{2} b_{23} & b_{33}
  \end{bmatrix}
  \right)
\]
where $1/2$ is a purely formal symbol and $a_{ij},b_{ij} \in R$ are the coefficients of the resolvent map
\[
  \phi(x_1\txi_1 + x_2\txi_2 + x_3\txi_3) = \sum_{1 \leq i < j \leq 3} (a_{ij} \eta_1 + b_{ij} \eta_2)x_i x_j.
\]
We will characterize maximality of $Q_\pp$ in terms of the $a_{ij}$ and $b_{ij}$. The first simplification is applicable to rings of any rank.
\begin{lem}[cf.~\cite{B3}, Lemma 22]\label{lem:newring}
Let $R$ be a DVR with maximal ideal $\pp$, and let $Q$ be an $R$-algebra of rank $n$. If $Q$ is not maximal, then there exists $k\geq 1$ and a basis $x_1, x_2, \ldots, x_n = 1$ of $Q$ such that
\[
  Q' = R\<\pp^{-1}x_1, \ldots, \pp^{-1}x_k, x_{k+1}, \ldots, x_n\>
\]
is a ring.
\end{lem}
\begin{proof}
Let $Q_1 \supsetneq Q$ be a larger algebra. Since $Q_1$ is a finitely generated submodule of $Q_K = \bigcup_{i\geq 0} \pp^{-i}Q$, it is contained in some $\pp^{-r}Q$. Pick $r$ such that
\[
  Q_1 \subseteq \pp^{-r}Q \quad \text{but} \quad Q_1 \nsubseteq \pp^{-r+1}Q.
\]
Then $Q' = Q + \pp^{r-1}Q_1$ is a rank-$n$ algebra such that
\[
  Q \subsetneq Q' \subseteq \pp^{-1}Q.
\]
Choose a basis $\tilde x_1,\ldots, \tilde x_k$ for the $R/\pp R$-vector space $\pp Q'/\pp Q$, and complete it to a basis $\tilde x_1,\ldots, \tilde x_n$ for $Q/\pp Q$. Since $1 \notin \pp Q'$, we can arrange for $\tilde x_n = 1$. Then by Nakayama's lemma, any lifts $x_1,\ldots, x_n$ generate $Q$, and
\[
  \pp^{-1}x_1, \ldots, \pp^{-1}x_k, x_{k+1}, \ldots, x_n
\]
generate $Q'$.
\end{proof}
\begin{thm}[cf.~\cite{B3}, pp.~1357--58]
Let $Q$ be a quartic algebra over a DVR $R$ with maximal ideal $\pp$, and let $\phi : Q/R \to M$, $\theta : \Lambda^3(Q/R) \to \Lambda^2 M$ be a resolvent. Then $Q$ is non-maximal if and only if, under some choice of bases, the matrices $(A,B)$ representing $\phi$ satisfy one of the following conditions:
\begin{enumerate}[$(a)$]
  \item $\pp^2$ divides $a_{11}$, and $\pp$ divides $a_{12}$, $a_{13}$, and $b_{11}$.
  \item $\pp$ divides $a_{11}$, $a_{12}$, $a_{22}$, $b_{11}$, $b_{12}$, and $b_{22}$.
  \item $\pp^2$ divides $a_{11}$, $a_{12}$, and $a_{22}$, and $\pp$ divides $a_{13}$ and $a_{23}$.
  \item $\pp$ divides all $a_{ij}$.
\end{enumerate}
\begin{proof}
The basic strategy is to find a suitable extension of the resolvent map to the ring $Q'$ in Lemma \ref{lem:newring}, examining the cases where $k$ is $1$, $2$, and $3$.

First assume that $Q$ has content $1$ (by which we mean that the content ideal $\ct(Q)$ is the whole of $R$). Then $k$ is $1$ or $2$ and $Q'$ also has content $1$. Both $Q$ and $Q'$ have unique (minimal and numerical) resolvents $(M,\theta,\phi)$ and $(M',\theta',\phi')$, where (since $Q_K = Q'_K$) we have $M \subseteq M' \subseteq M_K$, and $\theta$ and $\phi$ are the restrictions of $\theta'$ and $\phi'$. Also, since $Q$ has index $\pp^k$ in $Q'$, $M$ has index $\pp^k$ in $M'$.

If $k = 1$, then we can arrange our coordinates such that
\begin{gather*}
  Q/R = \langle\txi_1,\txi_2,\txi_3\rangle, Q'/R = \langle\pi^{-1}\txi_1,\txi_2,\txi_3\rangle \\
  M = \langle\eta_1,\eta_2\rangle, M' = \langle\eta_1,\pi\eta_2\rangle.
\end{gather*}
Now since $\phi' : Q'/R \to M'$ is the extension of $\phi$, its corresponding matrix $(A',B')$ is given by a straightforward change of basis:
\[
  (A',B') = \left(
  \begin{bmatrix}
    \pi^{-2} a_{11} & \frac{1}{2} \pi^{-1} a_{12} & \frac{1}{2} \pi^{-1} a_{13} \\
    \frac{1}{2} \pi^{-1} a_{12} & a_{22} & \frac{1}{2} a_{23} \\
    \frac{1}{2} \pi^{-1} a_{13} & \frac{1}{2} a_{23} & a_{33}
  \end{bmatrix},
  \begin{bmatrix}
     \pi^{-1} b_{11} & \frac{1}{2} b_{12} & \frac{1}{2} b_{13} \\
    \frac{1}{2} b_{12} & \pi b_{22} & \frac{1}{2} \pi b_{23} \\
    \frac{1}{2} b_{13} & \frac{1}{2} \pi b_{23} & \pi b_{33}
  \end{bmatrix}
  \right)
\]
The entries of this matrix (sans $1/2$'s) must belong to $R$, giving the divisibilities listed in case (a) above.

If $k = 2$, then the proof works similarly, except that $M'$ takes one of the two forms $\<\pi\eta_1,\pi\eta_2\>$ and $\<\eta_1,\pi^2\eta_2\>$. We leave it to the reader to write out the corresponding matrices $(A',B')$ and conclude cases (b) and (c) above, respectively.

We are left with the case that $\ct(Q) \neq 1$, that is, there is a quartic ring $Q'$ with $Q = R + \pi Q'$. (A priori we might only have a ring $Q''$ with $Q = R + \pi^k Q''$, $k \geq 1$; but then $Q' = R + \pi^{k-1}$ has the aforementioned property.) Then we may select bases for $Q$ and $Q'$ in the form of Lemma \ref{lem:newring}, with $k = 3$. Since the resolvent is no longer unique, we must take care in choosing the new target module $M'$ of the resolvent $\phi'$. Since $\phi$ is quadratic and $Q'/R = \pi^{-1}(Q/R)$, a natural candidate is $M' = \pi^{-2} M$, but unfortunately this is too large: we have $[M' : M] = \pp^4$ but $[Q'/R : Q/R] = \pp^3$, so $\theta'$ cannot possibly be an isomorphism. However, since $\ct(Q) \neq 1$, we have $\phi(Q/R) \subsetneq M$, so picking a sublattice $L \subsetneq M$ of index $\pp$ containing $\phi(Q/R)$, we get that $M' = \pp^{-2}L$ yields a workable resolvent. Note that $\pp^{-2}M \subsetneq M' \subsetneq \pp^{-1}M$, so we can take a basis such that
\[
  M = \langle\eta_1,\eta_2\rangle \textand M' = \langle\pi^{-1}\eta_1, \pi^{-2}\eta_2\rangle.
\]
We then get
\[
  (A',B') = \left(
  \begin{bmatrix}
    \pi^{-1} a_{11} & \frac{1}{2} \pi^{-1} a_{12} & \frac{1}{2} \pi^{-1} a_{13} \\
    \frac{1}{2} \pi^{-1} a_{12} & \pi^{-1} a_{22} & \frac{1}{2} \pi^{-1} a_{23} \\
    \frac{1}{2} \pi^{-1} a_{13} & \frac{1}{2} \pi^{-1} a_{23} & \pi^{-1} a_{33}
  \end{bmatrix},
  \begin{bmatrix}
    b_{11} & \frac{1}{2} b_{12} & \frac{1}{2} b_{13} \\
    \frac{1}{2} b_{12} & b_{22} & \frac{1}{2} b_{23} \\
    \frac{1}{2} b_{13} & \frac{1}{2} b_{23} & b_{33}
  \end{bmatrix}
  \right),
\]
yielding condition (d).

Conversely, if one of the conditions (a)--(d) holds, the foregoing calculations suggest how to embed $L = Q/R$ and $M$ into lattices $L'$ and $M'$ with $L \subsetneq L'$, such that the extensions of $\theta$ and $\phi$ still form a resolvent, yielding a quartic ring $Q'$ that contains $Q$ as a proper subring.
\end{proof}
\end{thm}
\section{Conclusion}
We have found the Dedekind domain to be a suitable base ring for generalizing the integral parametrizations of algebras and their ideals by Bhargava and his forebears. In each case, ideal decompositions $\aa_1 \oplus \cdots \oplus \aa_n$ fill the role of $\ZZ$-bases, and elements of appropriate fractional ideals take the place of integers in the parameter spaces. We have also shown that the notion of ``balanced,'' introduced by Bhargava to describe the ideal triples parametrized by general nondegenerate $2\times 2\times 2$ cubes, has some beautiful properties and is worthy of further study. We expect that the methods herein will extend to replicate the other parametrizations in Bhargava's ``Higher Composition Laws'' series and may shed light on the analytic properties of number fields and orders of low degree over base fields other than $\QQ$.

A generalization to quintic algebras over a Dedekind domain, following \cite{B4}, has been found. Details are to appear in a forthcoming publication (see \url{arxiv.org/abs/1511.03162}).

\begin{backmatter}

\section*{Competing interests}
  I have no competing interests.


\section*{Acknowledgements}
A previous version of this paper served as my senior thesis at Harvard College. I thank my thesis advisor, Benedict Gross, for many helpful discussions and comments. I thank Melanie Wood for clarifications on the relationships between my work and hers. I thank Arul Shankar for useful discussions, especially for informing me that he and Wood had been interested in the question answered by Corollary \ref{cor:at least 1 res}. I thank Brian Conrad for the suggestion that I work with Prof.~Gross. I thank Noam Elkies and the editors for hunting down some subtle errors. I thank Ken Ono for suggesting RMS as a publication venue.


\bibliography{thesis_rms}
\bibliographystyle{bmc-mathphys}

\section*{Figures}
  \begin{figure}[h!]
\centering
\includegraphics[height=1.5in]{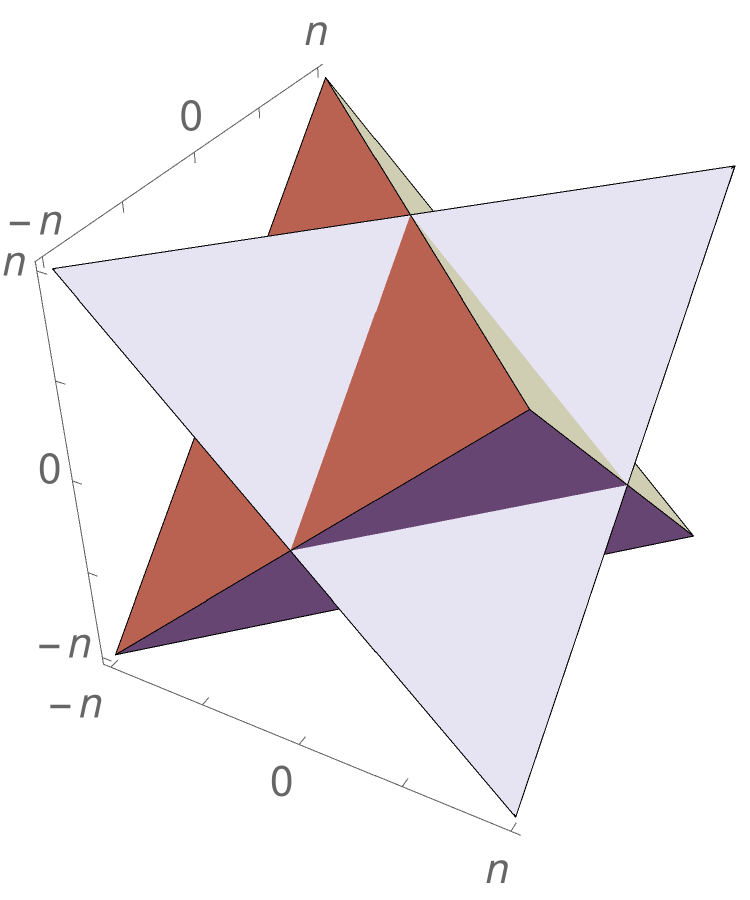}
\caption{\csentence{Stella octangula} showing the range of ideal triples in $\ZZ_p[p^n\sqrt{u}]$ that are balanced}
\label{fig:stella}
\end{figure}

\end{backmatter}

\end{document}